\documentclass[11pt,letterpaper]{amsart}
\pdfoutput=1
\usepackage{graphicx}
\usepackage{listings}
\usepackage{longtable}
\usepackage{amssymb, amsmath, amscd}
\usepackage{float}
\usepackage[mathscr]{euscript}
\usepackage{array}
\usepackage[bookmarks=false]{hyperref}
\usepackage{tikz-cd}
\usepackage{caption}
\usepackage{subcaption}

\newtheorem{theorem}{Theorem}[section]
\newtheorem{corollary}[theorem]{Corollary}

\newtheorem{proposition}[theorem]{Proposition}

\newtheorem{definition-proposition}[theorem]{Definition-Proposition}
\newtheorem{lemma-notation}[theorem]{Lemma-Notation}

\theoremstyle{definition}
\newtheorem{definition}[theorem]{Definition}

\newtheorem{example}[theorem]{Example}
\newtheorem{remark}[theorem]{Remark}

\newtheorem{notation}[theorem]{Notation}

\newcommand{\Z}{\mathbb{Z}}

\newcommand{\Q}{\mathbb{Q}}

\newcommand{\R}{\mathbb{R}}
\newcommand{\C}{\mathbb{C}}

\newcommand{\threepartdef}[6]
{
	\left\{
		\begin{array}{lll}
			#1 & \mbox{if } #2 \\
			#3 & \mbox{if } #4 \\
			#5 & \mbox{if } #6
		\end{array}
	\right.
}

\newcommand{\twobytwo}[4]
{
	\begin{pmatrix}
		#1 & #2 \\
		#3 & #4
	\end{pmatrix}
}

 \makeatletter\let\@wraptoccontribs\wraptoccontribs\makeatother
\newcommand{\triv}{\mathbf{1}} 
\newcolumntype{H}{>{\setbox0=\hbox\bgroup}c<{\egroup}@{}} 

\begin{document}
\title{Exact enumeration of fullerenes}

\author{Philip Engel}
\address[Philip Engel]{University of Bonn}
\email{engel@math.uni-bonn.de}

\author{Peter Smillie}
\address[Peter Smillie]{Heidelberg University}
\email{psmillie@mathi.uni-heidelberg.de}

\address[Jan Goedgebeur]{KU Leuven}
\email{jan.goedgebeur@kuleuven.be}

\contrib[With an appendix by]{Jan Goedgebeur}

\maketitle

\begin{abstract}
A fullerene, or buckyball, is a trivalent graph on the sphere with only pentagonal
and hexagonal faces. Building on ideas of Thurston,
we use modular forms to give an exact formula for the number of
oriented fullerenes with a given number of vertices. 
\end{abstract}

\section{Introduction}

A {\it triangulation} $\mathcal{T}$ of $S^2$ is an embedded graph in the two-dimensional 
sphere whose complementary faces are triangles.
We say that $\mathcal{T}$ is {\it convex} if every vertex has valence
$6$ or less. The {\it curvature} of a vertex $v_i\in \mathcal{T}$ is $\kappa_i:=6-{\rm valence}(v_i)$ 
and the {\it curvature profile} $\{\kappa_1,\dots,\kappa_m\}$ is the multiset of nonzero curvatures.
It is a consequence of Euler's formula that $\sum \kappa_i=12$ and so there are at most $12$
vertices of positive curvature. The dual complex of a triangulation with exactly $12$ vertices of
positive curvature is a {\it buckyball} or {\it fullerene}---a trivalent graph on the sphere,
whose faces are pentagons and hexagons.

The terminology originates from the famous Buckminsterfullerene molecule, with chemical formula ${\rm C}_{60}$.
It is formed from $60$ carbon atoms, each one bonded to three other carbon
atoms in $sp^2$-hybridized orbitals, and forming a spherical molecule with carbon
rings of length $5$ and $6$, see Figure \ref{fig:buckyball}. The molecule was first discovered in
the laboratory in 1985 by Kroto, Heath, O'Brien, Curl, and Smalley \cite{Kroto:1985aa},
work which received the 1996 Nobel prize in chemistry. 

\begin{figure}
\includegraphics[width=1in]{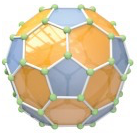}
\caption{The Buckminsterfullerene molecule}
\label{fig:buckyball}
\end{figure}

\subsection{Algorithmic enumeration}
Soon after the discovery of fullerenes, the question of their enumeration was
 explored by chemists, computer scientists, and mathematicians, resulting in a large body of literature: see
 Liu, Klein, Schmalz, and Seitz \cite{Liu:1991aa}, Manolopoulos and May \cite{Manolopoulos:1991aa},
 Manolopoulos and Fowler \cite{Manolopoulos:1992aa, Manolopoulos:1993aa},
 Sah \cite{Sah:1993aa}, Yoshida and Osawa \cite{Yoshida:1995aa}, Brinkmann and Dress \cite{BD97},
Hasheminezhad, Fleischner, and McKay \cite{Hasheminezhad:2008aa},  Brinkmann, Goedgebeur, and McKay \cite{BGM12,Goedgebeur:2015aa}.

The earliest algorithms, proposed by Manolopoulos {\it et al.~}enumerated fullerenes by spiraling outward from a face,
as one might peel an orange. A {\it spiral representation} $[i_1,\dots,i_{12}]$ of the fullerene is then simply the list of
$12$ indices corresponding to positions of the 12 pentagons in a face spiral of the given fullerene.
The {\it canonical spiral representation}  is the lexicographically minimal spiral representation.
For example, the canonical spiral representation
of Buckminsterfullerene is $$[1, 7, 9, 11, 13, 15, 18, 20, 22, 24, 26, 32].$$ Remarkably, there are fullerenes with
no spiral representation, the first counterexample being an isomer of ${\rm C}_{380}$ with tetrahedral
symmetry \cite{Brinkmann:2012ab}.


The spiral algorithm was supplanted by Brinkmann and Dress's 
algorithm, implemented in the program {\it fullgen}. It generates all fullerenes with a given number
of vertices by stitching together patches which are bounded by zigzag (Petrie) paths.
This was again supplanted by {\it buckygen}, developed
by Brinkmann, Goedgebeur, and McKay \cite{BGM12}.
It generates larger fullerenes from smaller ones,
by excising a patch of faces and replacing it with a larger patch having the same boundary. 
There are three irreducible fullerenes in the {\it buckygen} algorithm, isomers
of ${\rm C}_{20}$, ${\rm C}_{28}$, and ${\rm C}_{30}$.

Because of its recursion, {\it buckygen} is much more efficient than {\it fullgen},
especially for generating fullerenes up to $n$ vertices.
Contradicting results of {\it buckygen} and {\it fullgen} led to the detection
of a small programming error in {\it fullgen}, which led to a discrepancy in the counts beyond
$136$ vertices. After this was fixed, the two programs agreed up to 380
vertices, giving a high degree of confidence in their accuracy.

The input to this paper was generated with {\it buckygen}, see Appendix \ref{appendix} by Jan Goedgebeur,
with a built-in setting that counts enantiomorphic (mirror) fullerenes as distinct. This makes
 the mathematical enumeration somewhat less complicated; we hope to return
 to the case of enumerating unoriented fullerenes in future work.

\subsection{Mathematical enumeration}
Thurston's paper ``Shapes of Polyhedra" \cite{thurston} introduced the moduli spaces
$\mathcal{M}_{\alpha_1\,,\,\dots\,,\,\alpha_k}$ of
{\it flat cone spheres} (modulo scaling): These are flat metrics on $S^2$ with $k$
conical singularities $p_i$
of cone angles $2\pi-\alpha_i$ and $\alpha_i\in (0,2\pi)$. The Gauss-Bonnet formula
implies $\sum \alpha_i = 4\pi$.
Alexandrov's theorem \cite{alexandrov} states that
a flat cone sphere is isometric to a convex polyhedron in $\R^3$,
unique up to rigid motion. 

Thurston proved that $\mathcal{M}_{\alpha_1\,,\,\dots\,,\,\alpha_k}$ admits a canonical Hermitian metric,
endowing it with the structure of a locally symmetric space: It admits a local isometry \cite[Thm.~0.2]{thurston} to
complex hyperbolic space $$\mathbb{CH}^{k-3}:= \mathbb{P}\{v\in \C^{1,{k-2}}\,\big{|}\,v\cdot v>0\},$$
also called the complex ball. The metric completion $\overline{\mathcal{M}}_{\alpha_1\,,\,\dots\,,\,\alpha_k}$
is a complex hyperbolic cone manifold, and admits a moduli-theoretic interpretation,
corresponding to allowing collections of cone points $\{p_i\,\big{|}\,i\in I\}$
to coalesce, so long as $\sum_{i\in I} \alpha_i<2\pi$.

The completed Thurston moduli space of interest to this paper is
$$M:=\overline{\mathcal{M}}_{\pi/3\,,\,\dots\,,\,\pi/3}$$ as every convex triangulation defines
a point in it, by declaring each triangle metrically equilateral of side length $1$.
A convex triangulation corresponds to a fullerene if and only if the corresponding
flat cone sphere has $12$ distinct points of positive curvature, or equivalently, defines
a point in the open stratum $\mathcal{M}_{\pi/3\,,\,\dots\,,\,\pi/3}$.
There are $47$ possible curvature profiles for
convex triangulations of the two sphere
$S^2$, in bijection with the $47$ partitions $\{\kappa_1,\dots,\kappa_m\}$ of $12$ into parts
of size $5$ or less. Each such partition corresponds to a {\it curvature stratum}:
$$M=\overline{\mathcal{M}}_{\pi/3\,,\,\dots\,,\,\pi/3}=\!\!\!\!\!\!\coprod_{\substack{ \sum \kappa_i=12 \\ 
\kappa_i\in\{1,2,3,4,5\}}} \!\!\!\!\!\!\mathcal{M}_{\kappa_1\pi/3\,,\,\dots\,,\,\kappa_m\pi/3 }.$$

Thurston proved that $M=\mathbb{P}\Gamma\backslash \mathbb{CH}^9$
is a global ball quotient for an arithmetic subgroup $\Gamma\subset {\rm U}(1,9)$;
it is the largest-dimensional completed Thurston moduli space which is a global quotient.
See also \cite{schwartz}.

\begin{definition} An {\it Eisenstein lattice} $\Lambda$ is a finitely generated $\Z[\zeta_6]$-module
$\Lambda\simeq \Z[\zeta_6]^k$, together with a non-degenerate
$\Z[\zeta_6]$-valued Hermitian inner product $v\star w$.
Its {\it signature} $(m,n)$ is the signature of
 $\Lambda\otimes_{\Z[\zeta_6]}\C\simeq \C^{m,n}$.  \end{definition}

\begin{theorem}[{\cite[Thm.~0.1]{thurston}}] \label{thm:Thurston}
There is an Eisenstein lattice $\Lambda$ of signature $(1,9)$
for which convex triangulations of the sphere, up to oriented isomorphism,
are in bijection with $\Gamma\backslash \Lambda^+$. Here $\Lambda^+\subset \Lambda$
is the set of vectors of positive norm and
$\Gamma={\rm U}(\Lambda)$ is the group of unitary isometries. 

Furthermore, for a vector $v\in \Lambda^+$ corresponding to a triangulation,
the number of triangles is $\tfrac{2}{3}(v\star v)$.
\end{theorem}

The point in $\mathbb{P}\Gamma\backslash \mathbb{CH}^9 =M$
corresponding to the triangulation is the $\Gamma$-orbit of the complex line $\C v\in \mathbb{CH}^9$.

Thurston's theorem bounds the growth rate for the number of fullerenes with $2n$ vertices
as $O(n^9)$. The computational chemists were aware of this fact and Thurston's result,
see for instance the survey of Schwerdtfeger, Wirz, and Avery \cite{Schwerdtfeger:2015aa}.

In 2017, the authors used Thurston's theorem to enumerate convex triangulations, when counted with
an appropriate weight:

\begin{definition}\label{mass} Define the {\it mass} of a convex triangulation to be
$$m(\mathcal{T}):=\frac{1}{|{\rm Aut}^+(\mathcal{T})|}\prod_{i=1}^m\frac{(1-\kappa_i/6)^{\kappa_i-1}}{\kappa_i!}$$
where $\kappa_i$ are the curvatures and ${\rm Aut}^+(\mathcal{T})$ is the group of oriented
automorphisms. Alternatively, we can define $m(\mathcal{T})= {|{\rm Stab}_\Gamma(v)|}^{-1}$
where $v\in \Lambda^+$ corresponds to $\mathcal{T}$.
\end{definition}

\begin{theorem}[{\cite[Thm.~1.1]{es}}]\label{weighted-count} The mass-weighted sum of all convex
triangulations with $2n$ triangles is exactly $$\sum_{|\mathcal{T}| = 2n} m(\mathcal{T}) =
\frac{809}{2^{15}3^{13}5^2} \sum_{d\mid n}d^9.$$
\end{theorem}

The key point of the proof is to integrate the Siegel theta function 
over the quotient $\mathbb{P}\Gamma\backslash \mathbb{CH}^9$. This produces
a Maass form of weight $-8$, and successively applying either the scaled raising operator
$\tfrac{1}{2\pi i}\partial_\tau -\frac{k}{4\pi}({\rm Im}\,\tau)^{-1}$ or $q\partial_q$ nine times (these give
the same result by Bol's identity), one gets the
mass-weighted generating function $F(q)$ for convex triangulations. Furthermore, it is 
necessarily a modular form of weight $10$ for ${\rm SL}_2(\Z)$.
Up to a constant scaling factor, this pins down $F(q)$
uniquely as the Eisenstein series $E_{10}(q) = -\frac{1}{264}+\sum_{n\geq 1} \big( \sum_{d\mid n}d^9\big)q^n$.
To determine the constant, it suffices to compute the 
the weighted count of triangulations with $2$ triangles. 
%
%

Let $G$ be a group acting on a set $S$, with finitely many orbits
$\{G\cdot s\,\big{|}\,s\in S\}$, and finite
stabilizers. It is most natural to weight the count of orbits by
$|{\rm Stab}_G(s)|^{-1}$ where $s\in S$ is an orbit representative.
This way, passing to a finite index subgroup $H\subset G$ changes the
weighted count by the index $[G: H]$. It is this manner
of enumeration employed in Theorem \ref{weighted-count},
with $G={\rm U}(\Lambda)$ and $S= \Lambda^+$.

\subsection{Main results} 
The enumeration of fullerenes by chemists is {\it not} a mass-weighted count.
In contrast to the mass-weighted enumeration, we call an unweighted count of orbits
{\it naive enumeration}.

\begin{definition} Let ${\rm BB}(q) := \sum_{n\geq 1} {\rm BB}_nq^n$ where
${\rm BB}_n$ is the naive count of oriented fullerenes with $2n$ vertices. We have
$${\rm BB}(q) = q^{10}+q^{12}+q^{13}+3q^{14}+3q^{15}+10q^{16}+9q^{17}+23q^{18}+\cdots.$$
\end{definition}

Obviously, ${\rm BB}(q)\in \Z[[q]]$ has integer coefficients, while the weighted generating
function for convex triangulations $F(q)\in \Q[[q]]$ does not. Even though the $q^n$-coefficients of 
$F(q)$ and ${\rm BB}(q)$ as asymptotic to each other as $n\to \infty$,
the series differ in two key ways:
\begin{enumerate}
\item Counts for $F(q)$ are over the {\it completed} space
$M=\overline{\mathcal{M}}_{\pi/3\,,\,\dots\,,\,\pi/3}$ whereas counts for ${\rm BB}(q)$
are over the open stratum $\mathcal{M}_{\pi/3\,,\,\dots\,,\,\pi/3}$. \vspace{2pt}
\item Counts for $F(q)$ are weighted, but counts for ${\rm BB}(q)$ are naive. Both
$\{\kappa_1,\dots,\kappa_m\}$ and ${\rm Aut}^+(\mathcal{T})$ affect the weight.
\end{enumerate}

Our goal is to surmount these obstacles, and count 
fullerenes naively. While conceptually simpler, this is mathematically
much more complicated.

The main idea is to prove that for each curvature profile and automorphism group,
there is an appropriately weighted count of such triangulations which is again
a (quasi)modular form, but of lower weight and higher level. 
Then, by M\"obius inversion on strata, the naive count ${\rm BB}(q)$ has an expression
as a weighted inclusion-exclusion of weighted counts.
So ${\rm BB}(q)$ is a linear combination of modular forms,
but of mixed levels and weights. Sah \cite[Thm.~3.4]{sah2} was the first
to realize ${\rm BB}_n$ would have a formula of an arithmetic nature.

Our formula is proven by bounding the levels and weights
of these modular forms explicitly, and using the {\it buckygen} program as input---we
solve a large linear system for the coefficients of the linear combination. Fortuitously,
 {\it buckygen} can compute $200$ Fourier
coefficients ${\rm BB}_n$. This exceeds the $152$-dimensional space of
modular forms in which we prove ${\rm BB}(q)$ lives.
Once a formula in a finite-dimensional space is established,
the ambient space of modular forms is pruned, resulting in our main theorem:

\begin{theorem} ${\rm BB}(q)$ is an explicit linear combination of Eisenstein series,
of weight at most $10$ and levels lying in $\{1,2,3,4,5,6,9,12,15,18\}.$ \end{theorem}

More precise statements are given in the body
of the paper, see Theorem \ref{main}. Using this explicit formula, we can compute the number
of oriented fullerenes far beyond {\it buckygen}. In fact, evaluating the formula is computationally
only as hard as factoring $n$.

\section{Strata of triangulations}

We must understand the stratification of 
$M=\overline{\mathcal{M}}_{\pi/3\,,\,\dots\,,\,\pi/3}=\mathbb{P}\Gamma\backslash \mathbb{CH}^9$
by curvatures and automorphism groups in more detail. 

\begin{definition} The {\it automorphism stratification} of $M$
is the stratification by oriented automorphism groups ${\rm Aut}^+(\mathcal{T})$.
Let $M_G\subset M$ denote the stratum of flat
cone spheres with ${\rm Aut}^+(\mathcal{T})=G$. So $M=\coprod_G M_G$. \end{definition}

A flat cone sphere can be specified by a {\it Lauricella differential}, i.e.~a multivalued 
differential $\omega$ on $\mathbb{P}^1$ with fractional pole orders (see \cite{dm, looijenga}). 
Such a differential defines a flat structure $\mathcal{T}$ on $\mathbb{P}^1\simeq S^2$ by declaring
$z$ a local flat coordinate if $dz=\omega$, so that the flat metric is $|\omega|^2$. Two Lauricella
differentials $\omega_1$ and $\omega_2$ determine the same flat cone sphere,
mod scaling, if and only if $\omega_1 = c \omega_2$ for $c\in \C^*$.
Conversely, every flat cone sphere is naturally a Riemann surface isomorphic
to $\mathbb{P}^1$, and defines a Lauricella differential by declaring $\omega =dz$ for 
a local flat coordinate $z$. 

Given any 12 points $p_i \in \mathbb{P}^1$, no more than 5 of which are
simultaneously equal, and any constant $c \in \C^*$, we form the Lauricella differential
\begin{equation}\label{eqn:lauricella}
\omega = c(x-p_1)^{-1/6}\cdots (x-p_k)^{-1/6}\,dx
\end{equation}
in which the product is over $p_i \neq \infty$. Then the resulting flat
cone sphere lies in $M$, and conversely every point of $M$ can be described by
a Lauricella differential of this form with $c = 1$. Thus, $M$ also admits a description as the
space of $12$ unordered points on $\mathbb{P}^1$, in which up to $5$ points can
coalesce, up to the action of the holomorphic automorphism group of $\mathbb{P}^1$. 

Note that such a differential
$\omega$ is multivalued and that $\omega = \zeta \omega$ for $\zeta$ any sixth root of unity.

\begin{proposition} The possible automorphism groups ${\rm Aut}^+(\mathcal{T})$ 
for some $\mathcal{T}\in M$ are precisely:
\begin{enumerate}
\item a cyclic group $C_n$ for $n\in \{1,2,3,4,5,6,7,8,9,10,11\}$,
\item a dihedral group $D_n$ for $n\in \{2,3,4,5,6,8,10,12\}$, or 
\item one of the three exceptional groups $T$, $O$, $I$, which are respectively
symmetries of the tetrahedron, octahedron, icosahedron.
\end{enumerate}
\end{proposition}

\begin{proof} The three exceptional groups are realized by putting $p_1,\dots,p_{12}$
at the vertices of the tetrahedron, octahedron, and icosahedron respectively. Note that for $T$, three 
points coincide at each vertex, and for $O$, two points coincide at each vertex.

The remaining finite subgroups of ${\rm PGL}_2(\C)$ are $C_n$ and $D_n$. The orbits
of the $C_n$ action, acting by $x\mapsto \zeta_n x$, are the singleton poles
$\{0\}$, $\{\infty\}$ and free orbits of size $n$. Since no more than $5$ points can coalesce
at $0$ or $\infty$, there must be at least one free orbit and so $n\leq 12$. If $n=12$, then
the points $p_i$ lie in a single free orbit of $C_{12}$, in which case
$D_{12}$ symmetry is automatic. For $C_n$ with $n\leq 11$, we can always put some points
$p_i$ at the poles $0$, $\infty$ with non-equal weights less than $6$, and add a union of free orbits,
to get a flat cone sphere $\mathcal{T}\in M$ with cyclic $C_n$
symmetry. 

To realize $D_n$ symmetry with $n\in\{2,3,4,5,6\}$, we place $6-n$ points at each of $0$, $\infty$
and place the remaining $2n$ points on a free $D_n$ orbit. To realize $D_n$ symmetry with $n\in \{8,10,12\}$,
we place $\frac{12-n}{2}$ points at each of $0$, $\infty$ and put the remaining $n$ points 
along the equator in a $D_n$ orbit of size $n$. The other dihedral groups $D_n$ for
$n\in \{7,9,11,13,14,\dots\}$ are not realizable because either they contain $C_n$ with $n\geq 13$
or the weights at $0$, $\infty$ cannot be made equal for parity reasons.
\end{proof}

Not all automorphism groups of flat cone spheres in $M$ are realized as automorphism
groups of triangulations:

\begin{example}\label{c11} The cyclic group $C_{11}$ acts on $\mathbb{P}^1$ by rotations,
with a generator acting by $x\mapsto \zeta_{11}x$. There is a flat cone sphere 
with one point at $\infty$ and $11$ points on the equator, all of curvature $1$,
associated to the Lauricella differential $(1-x^{11})^{-1/6}\,dx$. This flat cone sphere
is not realized by an equilateral triangulation.
\end{example}

If $\omega$ is a Lauricella differential of the form $\eqref{eqn:lauricella}$, then the
integral of $\omega$ along any path is well-defined up to a sixth root of unity.
By the {\it periods} of $\omega$, we mean the integrals along paths between
singular points, considered as numbers in $\C/\langle \zeta_6 \rangle$.
It follows from \cite{thurston} that:

\begin{proposition}
Equilaterally triangulated flat cone spheres are equivalent to Lauricella
differentials of the form \eqref{eqn:lauricella}, all of whose periods lie in $\Z[\zeta_6]$.
\end{proposition}

\begin{proof}
If $\mathcal{T}$ is triangulated, we choose the Lauricella
differential $\omega$ so that the period map sends
any given triangle of $\mathcal{T}$ onto a triangle of the standard
triangulation of $\C$. (Recall that $\omega$ is a priori determined
only up to a complex number of norm one).
Then the developing map sends all singularities of $\mathcal{T}$
onto points of $\Z[\zeta_6]$, so the periods of $\omega$ will lie in $\Z[\zeta_6]$.

In the other direction, if all the periods of $\omega$ lie in $\Z[\zeta_6]$, 
then we recover the triangulation by declaring a point $p \in \mathcal{T}$ to be a 
vertex if the integral of $\omega$ from any singularity to $p$ lies in $\Z[\zeta_6]$, 
and a path from one vertex to another to be an edge if the integral of $\omega$ 
along that path takes values between 0 and 1, up to the indeterminacy of $\langle \zeta_6 \rangle$.
\end{proof}

\begin{proposition}\label{prop-triangles} Suppose that a flat cone sphere $\mathcal{T}$ admits an equilateral
triangulation. Then ${\rm Aut}^+(\mathcal{T})$ preserves the triangulation.\end{proposition}

\begin{proof}
Let $\omega = c\prod_i (x-p_i)^{-1/6}\,dx$ be the Lauricella differential associated to
the triangulation. Oriented automorphisms of $\mathcal{T}$ are exactly given by $g\in {\rm PGL}_2(\C)$ for which 
$g^*\omega \in \C\omega$ or equivalently $g^*\{p_i\}=\{p_i\}$.
Since $g$ has finite order, we necessarily have 
$g^*\omega = \zeta\omega$ for some root of unity $\zeta$, well-defined up to a sixth root 
of unity. When $\mathcal{T}$ is triangulated, then
$g$ preserves the triangulation iff $\zeta = 1$
(up to a sixth root of unity).

Let $\gamma$ be any path between singularities of $\mathcal{T}$ with nonzero period. 
Since
$$\int_{g^*\gamma} \omega = \int_\gamma g^* \omega = \zeta \int_\gamma \omega$$
and all periods of $\omega$ are in $\Z[\zeta_6]$, we see that $\zeta \in \mathbb{Q}[\zeta_6]$. But the only roots of unity in $\Q[\zeta_6]$ are sixth roots of unity, so $g$ must preserve the Lauricella differential, and hence the triangulation.
 \end{proof}

If $\mathcal{T}$ is a flat cone sphere, and $G \subset \mathrm{Aut}^+(\mathcal{T})$,
then the quotient can also be given the structure of a flat cone sphere, where the cone
angle at a point in the quotient is the cone angle upstairs divided by the size of its stabilizer.
However, we will also find it useful to consider the quotient as an orbifold. To that end,
we define:
\begin{definition}
A $\emph{flat cone orbifold}$ is a flat metric on $S^2$ with $k$ conical
singularities of cone angle $2\pi - \alpha_i$ and specified ``ramification orders" $n_i$
subject to the condition that $n_i(2\pi - \alpha_i) \in (0, 2\pi]$ and $\alpha_i > 0$.
\end{definition}
A local orbifold chart near such a cone point is given by an $n_i$-to-$1$ branched
cover from a cone of angle $n_i(2\pi - \alpha_i)$, which is a local
isometry away from the branch point.

\begin{definition} An {\it orbitriangulation} of a flat cone orbifold is a system of compatible
equilateral triangulations in every possible orbifold chart.
\end{definition}

For the orbifold $\mathcal{T}/G$, an orbitriangulation of $\mathcal{T}/G$ is the same
thing as a $G$-invariant triangulation of $\mathcal{T}$. By Proposition \ref{prop-triangles},
every triangulation of $\mathcal{T}$ is $G$-invariant, so in particular we have

\begin{corollary}\label{orbi-cor} The quotient of a triangulated flat cone sphere
$\mathcal{T}\in M$ by a subgroup of its
automorphism group ${\rm Aut}^+(\mathcal{T})$ is an orbitriangulated flat cone sphere in
$M$.
\end{corollary}

\begin{figure}
     \centering
     \begin{subfigure}{0.49\textwidth}
         \centering
         \includegraphics[scale=2]{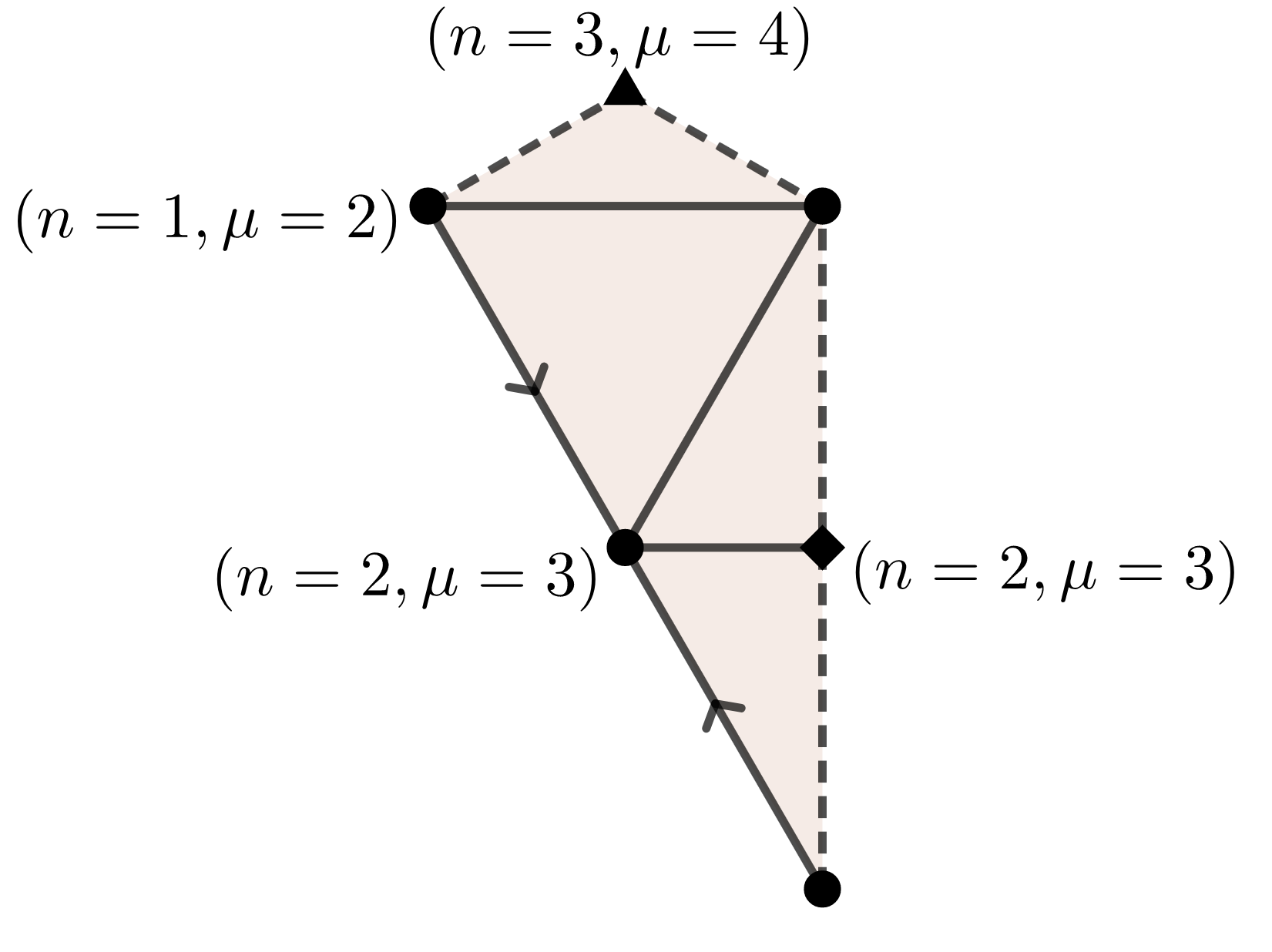}
     \end{subfigure}
     \begin{subfigure}{0.49\textwidth}
         \centering
         \includegraphics[scale=2.1]{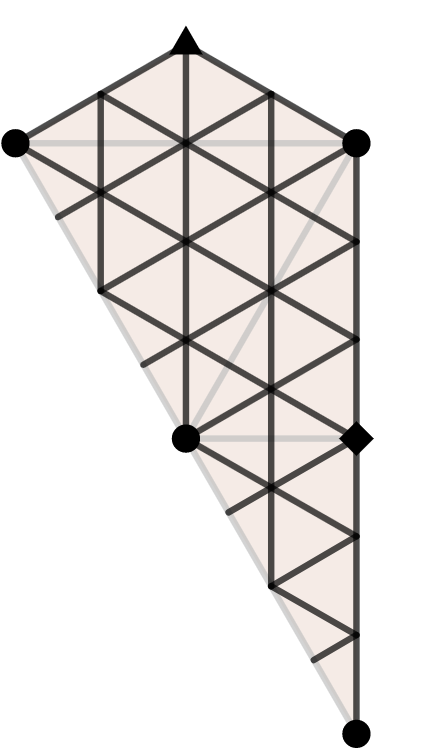}
     \end{subfigure}
        \caption{An orbitriangulation (left) and a triangulation of the
        underlying flat cone sphere.}
        \label{triangulated orbifold}
\end{figure}


\begin{example} Figure \ref{triangulated orbifold} (left) shows an orbitriangulation.
The points are labeled with their ramification orders $n$ and their
curvature $\mu$, which is related to the cone angle deficit $\alpha$ by
$\alpha/(2\pi) = \mu/6$. The unlabeled vertices are identified with the $(n=1, \mu=2)$
vertex by the gluing, which identifies the dotted lines in pairs and the solid lines
as marked. As required, we have $n(6-\mu) \leq 6$ at each point. The triangle and
diamond indicate points which are orbipoints but not vertices of the orbitriangulation.

Since the orbifold indices are $(3,2,2)$, this is equal to
$\mathbb{P}^1/D_3$ as an orbifold. The corresponding $D_3$-invariant triangulation
upstairs has $6$ singular points $p_i$ each of curvature $\kappa_i=2$, the preimages
of the $(n=1, \mu=2)$ point. In this example, none of the orbifold points are singular upstairs.
\end{example}

\begin{remark} A flat cone orbifold is not necessarily the quotient of a manifold
(i.e. a \emph{good orbifold}). For example, in the same picture of Figure
\ref{triangulated orbifold}, if we had declared the vertex between the two
arrows to be $(n=1, \mu = 3)$, this would be an orbitriangulated flat cone
orbifold that is not good. By the inequality $n_i(2\pi - \alpha_i) \leq 2\pi$, every good flat cone
orbifold is the quotient either of the plane, if $n_i(2\pi - \alpha_i) = 2\pi$ for all $i$,
or of a flat cone sphere.
\end{remark}

Note that an orbitriangulated flat cone sphere also admits a triangulation, but by triangles of side length
$\tfrac{1}{2\sqrt{3}}$ rather than $1$. See Figure \ref{triangulated orbifold} (right).

\begin{remark} As we can see from Example \ref{c11}, if $\mathcal{T}$ is not triangulated,
then its quotient by its automorphism
group may not even lie in $M$. For example, the quotient by the
$C_{11}$ action on the flat cone sphere $(1-x^{11})^{-1/6}\,dx$
lies in $\mathcal{M}_{61\pi/33\,,\,20\pi/11\,,\,\pi/3}$.\end{remark}

\begin{proposition}\label{poset-prop} The poset diagram of automorphism groups of convex triangulations is: 
 \begin{figure}[H]\label{aut-poset} \begin{tikzcd}
 & & C_1 \arrow[dash]{dl} \arrow[dash]{d} \arrow[dash]{ddr} &  \\
& C_2 \arrow[dash]{d} \arrow[dash]{dr} & C_3 \arrow[dash]{d} \arrow[dash]{ddl} & \\
 & D_2 \arrow[dash]{dl} \arrow[dash]{d} \arrow[dash]{dr} & D_3\arrow[dash]{d} \arrow[dash]{ddl} \arrow[dash]{ddr} & D_5  \arrow[dash]{dd} \\ 
  D_4  \arrow[dash]{dr} & T \arrow[dash]{d} \arrow[dash]{drr}  & D_6 &  \\
   & O  & & I
\end{tikzcd}
\end{figure}

\end{proposition}

\begin{proof} We first observe that any element $g\in {\rm Aut}^+(\mathcal{T})$ has order at most $6$.
This is because any $g\in {\rm PGL}_2(\C)$ fixes some point $p$. Assuming $g\neq 1$, 
either $p$ is the center of a triangular face and
${\rm ord}(g)= 3$, or $p$ is the center of an edge and ${\rm ord}(g)=2$, or $p$ is a vertex
and ${\rm ord}(g)$ divides ${\rm val}(p)\leq 6.$ 

Conversely, we can easily construct convex triangulations on which $C_n$ and $D_n$ act
for $n\leq 6$, for instance $n$-gonal bipyramids, and there are triangulated platonic solids with
automorphism groups $T$, $O$, $I$.

But, note that $C_4$, $C_5$, $C_6$ do not appear in the diagram. The reason is the following:
Any cyclic symmetry of order $4$, $5$, $6$ must fix a vertex of the triangulation
of valence $4$, $5$, $6$, at both of the poles of the cyclic rotation. There are either
one or two remaining free $C_4$, $C_5$, $C_6$ orbits.
Thus $C_4$, $C_5$, $C_6$ symmetry always implies $D_4$, $D_5$, $D_6$
symmetry---there is always an involution switching the two poles and the remaining
free orbit(s). 
\end{proof}

We now analyze how automorphism groups and curvatures interact.

\begin{definition} The {\it bistratification} of $M=\mathbb{P}\Gamma\backslash \mathbb{CH}^9$
is the stratification by conjugacy classes of inertia groups ${\rm Stab}_{\mathbb{P}\Gamma}(p)$ where $p\in \mathbb{CH}^9$. 
\end{definition}

\begin{proposition}\label{inertia-strata} The bistrata are the
intersections of automorphism strata and curvature strata. \end{proposition}

\begin{proof} The local stabilizer of some $[\mathcal{T}]\in M$ corresponding to
$\C v\in \mathbb{CH}^9$ has a canonical extension structure (see \cite{thurston}, discussion after Theorem 4.1)
$$1\to \textstyle \prod_i B_{\kappa_i}
\to {\rm Stab}_{\mathbb{P}\Gamma}(\C v) \to {\rm Aut}^+(\mathcal{T})\to 1$$
where $B_{\kappa_i}$ is the image in ${\rm U}(\Lambda)$ of the local braid group of the $\kappa_i$ colliding singularities.
Thus, ${\rm Stab}_{\mathbb{P}\Gamma}(p)$ jumps in size exactly when
(1) singularities further collide or (2) the size of the automorphism group increases.\end{proof}

\begin{definition} A bistratum is {\it triangulable}
if it contains a triangulation. \end{definition}

We henceforth consider only triangulable bistrata, unless otherwise stated.

A bistratum of $M$ is specified by the list of pairs $\{(n_i, \mu_i)\}$
of ramification orders and curvatures, at each of the singularities of the quotient 
flat cone orbifold.
Such a list determines a nonempty bistratum if and only if
\begin{enumerate}
\item The multiset of ramification orders $\{n_i\}$ not equal to 1
are those of a good orbifold $\mathbb{P}^1/G$,
for one of the groups $G$ in Proposition \ref{poset-prop},
\item for each $i$, $0<\mu_i <6 $ and $n_i(6 - \mu_i) \leq 6$, and 
\item $\sum_i \mu_i = 12$.
\end{enumerate}

The automorphism group $G$ of the corresponding bistratum is determined by
the $\{n_i\}$. The curvature profile $\{\kappa_j\}$ of the flat cone sphere upstairs is determined as follows:
Each singular point $p_i$ of the quotient orbifold with indices $(n_i, \mu_i)$ satisfying 
 $n_i(6-\mu_i)<6$ corresponds to $|G|/n_i$ singular points $\{q_j\}$ upstairs,
 each with curvature $\kappa_j$ given by the formula
 $$6-\kappa_j = n_i(6-\mu_i).$$ The singular points of the
 quotient orbifold with $n_i(6 - \mu_i) = 6$ are non-singular upstairs.

\begin{notation} 
We introduce a new notation for the bistratum $\{(n_i, \mu_i)\}$.
First, instead of listing the ramification orders $\{n_i\}$, we write the
corresponding group $G$. Second, we order the $\mu_i$ so that
they correspond to the ramification orders of $\mathbb{P}^1/G$ in descending
order. Third, we add a superscript
$\bf e$ to the index $\mu_i = 3$ whenever $n_i = 2$, and we add a superscript
$\bf f$ to the index $\mu_i = 4$ whenever $n_i = 3$.

So, for example, the bistratum
$\{(3,4), (2,3), (2,3), (1,2)\}$ of the orbifold depicted in Figure
\ref{triangulated orbifold} is written
$(D_3, \{4^{\bf f}, 3^{\bf e}, 3^{\bf e}, 2\})$. We will call the list
$\{\mu_i\}$ with superscripts a \emph{decorated} curvature profile.
\end{notation}

The superscript $\bf f$
stands for \emph{face}, because a singular point of an orbifold with $\mu_i = 4$ and
$n_i = 3$ may lie in the middle of a face of the orbitriangulation. The superscript
$\bf e$ stands for \emph{edge}, because a point of the orbifold with $\mu_i = 3$ and 
$n_i = 2$ may lie in the middle of an edge of the orbitriangulation.
Of course, both points may also lie at a vertex.

On the other hand, a singular point without a 
superscript must lie at a vertex. This is because first, the only
possible non-vertex points in a triangulation, with non-trivial stabilizers under an automorphism 
 group, are the center of an edge and the center of a face, 
 with stabilizers of size $2$ and $3$ respectively. So $n$ must be $2$ or $3$. Second, if $n(6-\mu) \neq 6$, the point upstairs has positive curvature 
 and hence must be a vertex.

\begin{remark}  \label{justify}
The justification for this notation is that
 the flat cone spheres underlying an orbitriangulation
in a given stratum of flat cone orbifolds depend
only on the decorated curvature profile---that is, which $\mu_i$
are decorated with an ${\bf f}$ or ${\bf e}$,
see Proposition \ref{curve-sublattice} for a precise statement.

The strata of flat cone orbifolds do depend,
in a weak way, on the group $G$. This is because $G$ dictates which
points of positive curvature are declared orbipoints. For example,
the strata $(C_2,\{4,4,4\})$ and $(C_1,\{4,4,4\})$ of flat cone orbifolds
have the same decorated curvature profile, so the underlying
flat cone spheres parameterizing orbitriangulations are the same.
But in $(C_2,\{4,4,4\})$, the first two points of positive curvature
are distinguished from the third, being orbipoints of order $2$,
whereas in $(C_1,\{4,4,4\})$, all three points are on equal footing.
\end{remark} 

\begin{definition} The {\it undecorated profile} $\mu^u$ of a quotient profile $\mu$ is the result 
  of deleting the decorations ${\bf f}$, ${\bf e}$. \end{definition}
  
  Undecorating corresponds to forgetting the orbifold structure, i.e.~taking the coarse space or
  resetting $n_i=1$. The orbitriangulations with curvature profile $\mu^u$ are simply triangulations
  in the curvature stratum $\mu^u$. 
  
  \begin{proposition}\label{curve-sublattice} Let $\kappa$ be an (undecorated) curvature 
  profile. There is a Hermitian Eisenstein sublattice $\Lambda_\kappa\subset \Lambda$ of signature
  $(1,{\rm len}(\kappa)-3)$ and a subgroup $\Gamma_\kappa\subset {\rm U}(\Lambda_\kappa)$
  for which $\Gamma_\kappa\backslash \Lambda_\kappa^+$ is in bijection with convex triangulations
  of the sphere, whose curvature profile is a collision of $\kappa$.
  
  Now let $(G,\mu)$ denote a group and
   decorated curvature profile and let $\mu^u$ be the undecoration of $\mu$. There is an Eisenstein
  module $$\Lambda_{\mu^u}\subset \Lambda_\mu\subset \tfrac{1}{2(1+\zeta_6)}\Lambda_{\mu^u}$$ 
  of signature $(1,{\rm len}(\mu)-3)$ and a subgroup 
  $\Gamma_{G,\mu}\subset {\rm U}(\Lambda_\mu)$ for which
   $\Gamma_{G,\mu}\backslash \Lambda_\mu^+$ is in bijection with
    orbitriangulations with covering group $G$ and profile a collision of $\mu$. The group $\Gamma_{G,\mu}$ is a finite
    index subgroup of $\Gamma_{\mu^u}$.  \end{proposition}
   
   \begin{proof} First, we treat the undecorated case. In the local period coordinates
   $(\int_{p_j}^{p_{j+1}} \omega)\in \C^{1,9}$, the condition that $\kappa_i$ 
   singularities $p_1,\dots,p_{\kappa_i}$ have coalesced is locally described by the equation
    $$\textstyle \int_{p_j}^{p_{j+1}}\omega =0\textrm{ for }j=1,\dots,\kappa_i-1.$$ This is a 
    $\Z[\zeta_6]$-linear condition, and thus cuts out a sublattice $\Lambda_{\kappa_i}\subset \Lambda$ 
    exactly corresponding to triangulations with a point of curvature $\kappa_i$.
   These linear conditions can be imposed for each coalescence, and are independent, 
   giving a sublattice $$\Lambda_\kappa := \cap_i\, \Lambda_{\kappa_i} \subset \Lambda.$$
   In fact, it is shown in Thurston's paper that the completed moduli space 
   $\overline{\mathcal{M}}_{\kappa_1\pi/3\,,\,\dots\,,\,\kappa_m\pi/3}$ is a subball quotient 
   $\mathbb{P}\Gamma_\kappa\backslash \mathbb{CH}^{{\rm len}(\kappa)-3}
   \subset \mathbb{P}\Gamma\backslash \mathbb{CH}^9$. The bijection 
   between $\Gamma_\kappa\backslash \Lambda^+_\kappa$ and triangulations works
    in the same way as for $M$: the triangulation associated to $v\in \Lambda_\kappa^+$ 
    corresponds to $$\C v \in \mathbb{P}\{x\in \Lambda_\kappa\otimes_{\Z[\zeta_6]} \C,\big{|}\,x\star x>0\} =
     \mathbb{CH}^{{\rm len}(\kappa)-3}.$$
   
   We now consider an orbistratum $(G,\mu)$. The moduli space of flat cone orbifolds
   containing orbitriangulations with curvature profile a collision of $\mu$ is an intermediate cover of
   $\overline{\mathcal{M}}_{\mu_1\pi/3\,,\,\dots\,,\,\mu_r\pi/3}$ as in 
   the undecorated case. The intermediate cover arises due to the fact that some points $p_i$
   are additionally marked with an orbifold order $n_i$, see Remark \ref{justify}.
    Triangulations with curvature profile $\mu^u$ correspond to $\Gamma_{\mu^{u}}$ orbits
   of $v\in \Lambda_{\mu^u}^+$ by the previous paragraph. In period coordinates, admitting a 
   triangulation is simply the condition $\int_{p_j}^{p_{j+1}}\omega\in\Z[\zeta_6]\textrm{ for all }j$.
   
   Fix a cone point $p_1$ for which the corresponding curvature $\mu_0$ is undecorated. Admitting an 
   orbitriangulation amounts to the following slightly weaker period conditions:
   \begin{align}\begin{aligned}\label{weak-ints}
   \textstyle \int_{p_1}^{p_j}\omega & \in\tfrac{1}{1+\zeta_6}\Z[\zeta_6]\textrm{ if }\mu_j=4^{\bf f}, \\ 
\textstyle \int_{p_1}^{p_j}\omega & \in\tfrac{1}{2}\Z[\zeta_6]\textrm{ if }\mu_j=3^{\bf e}, \\ 
\textstyle \int_{p_1}^{p_j}\omega & \in \Z[\zeta_6]\textrm{ otherwise.} 
\end{aligned}\end{align} This is because $\frac{1}{1+\zeta_6}\Z[\zeta_6]$ is the superlattice of $\Z[\zeta_6]$
containing the centers of the standard triangular tiles of $\C$ (and the vertices),
while $\frac{1}{2}\Z[\zeta_6]$ is the superlattice of $\Z[\zeta_6]$ containing
the centers of the edges (and the vertices).

Conditions (\ref{weak-ints}) on period integrals describe a finite index enlargement 
$\Lambda_\mu\supset \Lambda_{\mu^u}$ which is contained in
$\frac{1}{2(1+\zeta_6)}\Lambda_{\mu^u}$.  Equality of orbitriangulations
is the same as equality of the underlying flat cone spheres, preserving the ramification orders $n_i$.
Thus, $\Gamma_{G,\mu}\subset \Gamma_{\mu^u}$ is a finite index subgroup. 
   \end{proof}

\begin{remark} We will identify properties of the Eisenstein 
lattices/modules $\Lambda_\kappa$ and $\Lambda_\mu$
  in greater detail in the following section. \end{remark}
   
We call $\Lambda_\mu$ an Eisenstein module, rather than an Eisenstein lattice,  for the following reason:
  Its natural Hermitian inner product $\star$ need not be valued in $\Z[\zeta_6]$. Unlike $\Lambda_{\mu^u}$,
  there may be vectors $v\in \Lambda_\mu$ for which $v\star v\notin 3\Z$, because an orbitriangulation
  need not have an even integer number of triangles.

 \begin{definition}\label{completed-gen} 
 Let $f_{G,\mu}(q)=c_0+ \sum_{n >0} c_nq^n $, where $n \in \tfrac{1}{6} \mathbb{Z}$,
 be the completed generating function for orbitriangulations in the orbistratum $(G,\mu)$:
 Concretely, for $n > 0$, $c_n$ is the count of orbitriangulations
 $v\in \Gamma_{G,\mu}\backslash \Lambda_\mu^+$ with $2n$ triangles, 
  weighted by $|{\rm Stab}_{\Gamma_{G,\mu}}(v)|^{-1}$. 
   
 We define the constant term $c_0$ in the next section.
\end{definition}

Note that by Proposition \ref{curve-sublattice}, the dependency of $f_{G,\mu}(q)$
on $G$ is only multiplication by a scalar corresponding to the index of
$\Gamma_{G,\mu}\subset\Gamma_{\mu^u}$.

\begin{example}\label{gen-ex} We define the important generating function 
$$E(q) := \tfrac{1}{6}\sum_{v\in \Z[\zeta_6]} q^{v\overline{v}}.$$
For any $0$-dimensional decorated profile, i.e.~${\rm len}(\mu)=3$, 
we have that $f_{G,\mu}(q)\in \C E(q^r)$ where $r\in \tfrac{1}{6}\Z$ is
the (possibly fractional) number of triangles in the smallest area orbitriangulation
$\mathcal{T}_0$. This is because every orbitriangulation
with decorated profile $\mu$ is the result of scaling 
the flat structure $\mathcal{T}_0$ by some nonzero Eisenstein integer.

As a binary theta series $E(q)$ is a holomorphic modular form of weight $1$ and level $3$.
We identify it with an Eisenstein series in Example \ref{eis-ex}.
%
\end{example}
 
  \begin{definition} Let ${\rm BB}_{(G,\mu)}(q)$ to be generating function for naive enumeration
 of triangulations in the bistratum $(G,\mu)$ i.e. triangulations with automorphism group $G$
 and decorated quotient profile $\mu$.
 \end{definition} 
 
 Note that here we are counting the upstairs triangulations, not the quotient
 orbitriangulations, so ${\rm BB}_{(G,\mu)}(q)\in \Z[[q]]$.
 
 \begin{proposition}\label{lin-combo} Let $\{M_{(H,\nu)}\}$ be the collection of all lower-dimensional
  bistrata in the closure $\overline{M}_{(G,\mu)}$ of a fixed bistratum. There exist constants 
  $c_{(H,\nu)}\in \Q$ and $c_0\in \R$ for which 
  $$f_{G,\mu}(q^{|G|}) = {\rm BB}_{(G,\mu)}(q) + c_0+\sum_{(H,\nu)} c_{(H,\nu)} {\rm BB}_{(H,\nu)}(q).$$
  \end{proposition}
 
 \begin{proof} By Corollary \ref{orbi-cor}, a triangulation $\mathcal{T}$ with automorphism group $G$ gives rise 
 to a flat cone sphere $\mathcal{T}/G$ admiting an orbitriangulation with decorated profile $\mu$. 
 But conversely, given a group $G$ and an orbitriangulation with decorated profile $\mu$, the $G$-cover
 $\mathbb{P}^1\to \mathbb{P}^1/G$ branched over the appropriate orbipoints gives
a triangulation with a $G$-action.
 
 Thus, triangulations in the open bistratum $M_{(G,\mu)}$ with $|G|n$ triangles are in bijection with
 orbitriangulations with decorated profile $\mu$ and $n$ triangles, whose $G$-cover has 
 automorphism group $G$. Furthermore, in both $f_{G,\mu}(q)$ and ${\rm BB}_{(G,\mu)}(q)$,
  these (orbi)triangulations are counted with weight $1$. 
 
 In fact, whenever the automorphism group of an orbitriangulation jumps, so does the automorphism group 
 of its $G$-cover, and the same holds for singularity collisions. Thus, the $G$-covers of orbitriangulations 
 which contribute a fixed mass to $f_{G,\mu}(q)$ lie in a disjoint union of bistrata $M_{(H,\nu)}$.
 The proposition follows. 
 \end{proof} 

\begin{remark} Proposition \ref{lin-combo} proves that the inertia strata for $\Gamma_{G,\mu}$ are unions of 
inertia strata for $\Gamma = U(\Lambda)$. However the two stratifications are not identical: a stratum for $\Gamma_{G,\mu}$ can decompose further into substrata 
for $\Gamma$. For instance, if one quotients
 a triangulation $\mathcal{T}$ by a non-normal subgroup of ${\rm Aut}^+(\mathcal{T})$, 
 the additional symmetries of $\mathcal{T}$ do not descend
 to symmetries of the orbitriangulation.
 \end{remark}


 
 \begin{corollary}\label{cor-bb-strata} ${\rm BB}_{(G,\mu)}(q)$ is 
 a linear combination of $f_{H,\nu}(q^{|H|})$ ranging over all bistrata 
 $M_{(H,\nu)}\subset \overline{M}_{(G,\mu)}$ (plus a constant). \end{corollary}
 
 \begin{proof} This follows from Proposition \ref{lin-combo}, which shows that 
  ${\rm BB}_{(G,\mu)}(q)$ and $f_{G,\mu}(q^{|G|})$ differ
 by a strictly triangular change-of-basis.    \end{proof}

\begin{theorem}\label{include-exclude} The naive generating function ${\rm BB}(q)$ for fullerenes is a linear
 combination of $f_{G,\mu}(q^{|G|})$, ranging over all bistrata $M_{(G,\mu)}\subset M$.  \end{theorem}

\begin{proof} This follows immediately from Corollary \ref{cor-bb-strata} as 
$${\rm BB}(q) = \sum_{\substack{ (G,\mu)\textrm{ s.t.} \\ \kappa \,=\, 1^{12}}} {\rm BB}_{(G,\mu)}(q)$$
 where $\kappa$ is the curvature profile of the
$G$-cover of an orbitriangulation with decorated curvature profile $\mu$. \end{proof}

The final computation for this section is the enumeration of all bistrata. This is accomplished relatively
easily by considering the generic flat cone quotient for each group: 

\begin{table}
\begin{tabular}{|c | c| c|}
\hline
Group $G$ & Orbifold $\mathbb{P}^1/G$ & Longest profile $\mu$ \\ 
\hline
$I$ & $\mathbb{P}_{5,3,2}$ & $\{5,4^{\bf f},3^{\bf e}\}$ \\

$O$ & $\mathbb{P}_{4,3,2}$ & $\{5,4^{\bf f},3^{\bf e}\}^*$ \\

$T$ &  $\mathbb{P}_{3,3,2}$ & $\{4^{\bf f}, 4^{\bf f}, 3^{\bf e},1\}$  \\

 $D_6$ & $\mathbb{P}_{6,2,2}$ & $\{5, 3^{\bf e}, 3^{\bf e}, 1\}$ \\

 $D_5$ & $\mathbb{P}_{5,2,2}$ & $\{5, 3^{\bf e}, 3^{\bf e}, 1\}$  \\

 $D_4$ & $\mathbb{P}_{4,2,2}$ & $\{5, 3^{\bf e}, 3^{\bf e}, 1\}^*$ \\

 $D_3$ & $\mathbb{P}_{3,2,2}$ & $\{4^{\bf f}, 3^{\bf e}, 3^{\bf e},1,1\}$ \\

 $C_3$ & $\mathbb{P}_{3,3}$ & $\{4^{\bf f}, 4^{\bf f}, 1^4\}$ \\

 $D_2$ & $\mathbb{P}_{2,2,2}$ & $\{3^{\bf e}, 3^{\bf e}, 3^{\bf e},1^3\}$ \\

 $C_2$ & $\mathbb{P}_{2,2}$ & $\{3^{\bf e}, 3^{\bf e},1^6\}$ \\

 $C_1$ & $\mathbb{P}^1$ & $\{1^{12}\}$ \\
\hline
\end{tabular}
\vspace{3pt}

\caption{Relevant orbifolds, with largest bistratum}
\label{table-orb}
\end{table}

\begin{remark} We explain the meaning of $*$ in Table \ref{table-orb}.
The cuboctahedron is another flat cone sphere in $M$
with $O$ symmetry, of maximal dimension (zero) for the group.
But it cannot be equilaterally triangulated.

Similarly,
the referee pointed out that there is a second bistratum of complex dimension
$1$ with symmetry group $D_4$ corresponding to flat cone spheres with one
orbit of size $8$ and one orbit of size $4$. But none of these are triangulable,
since they have a point of cone angle $\tfrac{\pi}{2}$ corresponding to the images
of the poles of the order $4$ rotation.\end{remark}

All bistrata not shown in Table \ref{table-orb} result from colliding some singularities of curvature $\mu_i=1$ either into each other
or the orbipoints of $\mathbb{P}^1/G$.
%
The results of such collisions are shown in Table \ref{bistrata}.
On occasion, the $G$-cover of an
orbitrianguation with specified $(G,\mu)$ automatically
has a larger automorphism group,
leading to some repeated bistrata in Table \ref{bistrata}.
This is indicated in the ``Description" column, and only
occurs for bistrata of dimension $0$ or $1$.

\begin{proposition} There are 96 triangulable bistrata of $M$.
\end{proposition}

\begin{proof} The proposition follows from direct enumeration. \end{proof}

{\footnotesize 
\begin{longtable}[H]{| c | c | c | c | c | c | c | c |}
\caption{Triangulable bistrata of $M$}\label{bistrata}
\endfirsthead
\hline
$G$ & $|G|$ & $n_i$ & $\mu_i$ & Weight & Level  & $\C f_{G,\mu}(q^{|G|})$ & Description \\
    \hline
$I$ & $60$ & ${5,3,2}$ & ${5,4^{\bf f},3^{\bf e}}$  & 1  & $30$ & $E(q^{10})$ &  icosahedron \\
\hline
$O$ & $24$ & ${4,3,2}$ & ${5,4^{\bf f},3^{\bf e}}$ & 1 & $12$ & $E(q^4)$ & octahedron  \\
\hline
$T$ & $12$ &  ${3,3,2}$ & ${4^{\bf f},4^{\bf f},3^{\bf e},1}$  & 2   &   $(6,12)$ &     &   \\
 & &  & ${4^{\bf f},4^{\bf f},4}$ & 1      & $12$  &   $E(q^4)$ &  $(O,\{5,4^{\bf f}, 3^{\bf e}\})$ \\
 & & & ${5,4^{\bf f},3^{\bf e}}$  & 1     & $6$ &   $E(q^2)$  & tetrahedron  \\
\hline
$D_6$  & $12$  & ${6,2,2}$ & ${5,3^{\bf e},3^{\bf e},1}$   &  $2$ (Q) & $6$ & $E_2(q^6)$  & $6$-tubes  \\
& & & ${5,4,3^{\bf e}}$  & $1$ & $18$ & $E(q^{6})$ & doubled hexagon  \\
\hline
$D_5$& $10$ & ${5,2,2}$ & ${5,3^{\bf e},3^{\bf e},1}$    & $2$ (Q)  & $5$ & $E_2(q^5)$ &  $5$-tubes     \\
& & & ${5,4,3^{\bf e}}$  & $1$ & $15$ & $E(q^{5})$ & pentagonal bipyramid  \\
\hline
$D_4$ & $8$ & ${4,2,2}$ & ${5,3^{\bf e},3^{\bf e},1}$   & $2$ (Q) & $4$ & $E_2(q^4)$ &  $4$-tubes     \\
& & & ${5,4,3^{\bf e}}$  & $1$ & $12$ & $E(q^{4})$ & $(O,\{5,4^{\bf f}, 3^{\bf e}\})$  \\
\hline
$D_3$ & $6$ & ${3,2,2}$ & ${4^{\bf f},3^{\bf e},3^{\bf e},1,1}$ & $3$ & $(6,6)$ &  & \\
& & & ${4^{\bf f},3^{\bf e},3^{\bf e},2}$ & $2$ (Q)  & $(6,6)$ &  & \\
& & &  ${4^{\bf f},4,3^{\bf e},1}$ & $2$ & $(6,6)$ &  & \\
& & & ${4^{\bf f},4,4}$ & $1$ & $18$ & $E(q^{6})$ &  $(D_6, \{5,4,3^{\bf e}\})$  \\
 & & & ${5,3^{\bf e},3^{\bf e},1}$ & $2$ (Q)  & $3$ & $E_2(q^3)$ & $3$-tubes \\
& & & ${5,4,3^{\bf e}}$ & 1  & $9$ & $E(q^3)$ & triangular bipyramid \\
& & & ${4^{\bf f},5,3^{\bf e}}$ & 1  & $3$ & $E(q)$ & doubled triangle \\
\hline
$C_3$ & $3$ & ${3,3}$ & ${4^{\bf f},4^{\bf f},1^4}$ & $4$ & $(3,3)$ &  & \\
& & & ${4^{\bf f},4^{\bf f},2,1,1}$  & $3$ & $(3,3)$ &  & \\
& & & ${4^{\bf f},4^{\bf f},2,2}$ & $2$ (Q)  & $(3,3)$ &  & $(D_3,\{4^{\bf f},3^{\bf e},3^{\bf e}, 2\})$  \\
& & & ${4^{\bf f},4^{\bf f},3,1}$ & $2$ & $(6,3)$ &  & \\
& & & ${4^{\bf f},4^{\bf f},4}$ & $1$ & $3$ & $E(q)$ &  $(D_3,\{4^{\bf f},5,3^{\bf e}\})$  \\
& & & ${5,4^{\bf f},1^3}$ & $3$ & $(3,3)$ &  & \\
& & & ${5,4^{\bf f},2,1}$ & $2$ (Q)  & $(3,3)$ &  &  \\
& & & ${5,4^{\bf f},3}$ & $1$ & $6$ & $E(q^2)$ &  $(T, \{5,4^{\bf f},3^{\bf e}\})$ \\
& & &  ${5,5,1,1}$ & $2$ (Q)  & $3$ & $E_2(q^3)$ &  $(D_3, \{5,3^{\bf e},3^{\bf e},1\})$ \\
& & & ${5,5,2}$ & $1$ & $9$ & $E(q^3)$ &  $(D_3, \{5,4,3^{\bf e}\})$ \\
\hline
$D_2$ & $4$ & ${2,2,2}$ & ${3^{\bf e},3^{\bf e},3^{\bf e},1^3}$& $4$ & $(2,4)$ &  & \\
& & & ${3^{\bf e},3^{\bf e},3^{\bf e},2,1}$  & $3$ & $(6,4)$ &  & \\
& & & ${3^{\bf e},3^{\bf e},3^{\bf e},3}$ & $2$ (Q) & $2$ & $E_2(q^2)$ & pillowcases \\
& & & ${4,3^{\bf e},3^{\bf e},1,1}$  & $3$ & $(6,4)$ &  & \\
& & & ${4,3^{\bf e},3^{\bf e},2}$  & $2$ (Q)  & $(6,4)$ &  & \\
& & & ${4,4,3^{\bf e},1}$ & $2$ & $(6,4)$ &  & \\
& & & ${4,4,4}$ & $1$ & $12$  & $E(q^4)$ &  $(O, \{5,4^{\bf f},3^{\bf e}\})$   \\
& & & ${5,3^{\bf e},3^{\bf e},1}$ & $2$ (Q)  & $2$ & $E_2(q^2)$ & $2$-tubes \\
& & & ${5,4,3^{\bf e}}$ & $1$ & $6$ & $E(q^2)$ & doubled rhombus \\
\hline
$C_2$ & $2$ & ${2,2}$ & ${3^{\bf e},3^{\bf e}, 1^6}$& $6$ & $(2,2)$ &  & \\
& & & ${3^{\bf e},3^{\bf e}, 2,1^4}$ & $5$ & $(6,2)$ &  & \\
& & & ${3^{\bf e},3^{\bf e}, 2,2,1,1}$& $4$ & $(6,2)$ &  & \\
& & & ${3^{\bf e},3^{\bf e}, 2,2,2}$ & $3$ & $(6,2)$ &  & \\
& & & ${3^{\bf e},3^{\bf e}, 3,1^3}$  & $4$ & $(2,2)$ &  & \\
& & & ${3^{\bf e},3^{\bf e}, 3,2,1}$ &$3$  & $(6,2)$ &  & \\
& & & ${3^{\bf e},3^{\bf e}, 3,3}$ & $2$ (Q) & $2$ & $E_2(q^2)$ &   $(D_2, \{3^{\bf e},3^{\bf e},3^{\bf e},3\})$ \\
& & & ${3^{\bf e},3^{\bf e}, 4,1,1}$  & $3$ & $(6,2)$ &  & \\
& & & ${3^{\bf e},3^{\bf e}, 4,2}$ & $2$ (Q) & $(6,2)$ &  & \\
& & & ${3^{\bf e},3^{\bf e}, 5,1}$  & $2$ (Q) & $1$ & $E_2(q)$ & $1$-tubes \\
& & & ${4,3^{\bf e}, 1^5}$  & $5$ & $(6,2)$ &  & \\
& & & ${4,3^{\bf e}, 2,1^3}$  & $4$ & $(6,2)$ &  & \\
& & & ${4,3^{\bf e}, 2,2,1}$ & $3$ & $(6,2)$ &  & \\
& & & ${4,3^{\bf e}, 3,1,1}$  & $3$ & $(6,2)$ &  & \\
& & & ${4,3^{\bf e}, 3,2}$  & $2$ (Q)  & $(6,2)$ &  &  \\
& & & ${4,3^{\bf e}, 4,1}$  & $2$ & $(6,2)$ &  & \\
& & & ${4,3^{\bf e}, 5}$ & $1$ & $3$ & $E(q)$ & doubled pediment \\
& & & ${4,4, 1^4}$  & $4$ & $(3,2)$ &  & \\
& & & ${4,4, 2,1,1}$  & $3$ & $(3,2)$  &  & \\
& & & ${4,4, 2,2}$  & $2$ (Q) & $(3,2)$ &  & $(D_2, \{4, 3^{\bf e}, 3^{\bf e}, 2\})$ \\
& & & ${4,4, 3,1}$  & $2$ & $(6,2)$ &  & \\
& & & ${4,4,4}$ & $1$ & $6$ & $E(q^2)$ & $(D_2,\{5,4,3^{\bf e}\})$ \\
& & & ${5,3^{\bf e},1^4}$ & $4$ & $(2,2)$ &  & \\
& & & ${5,3^{\bf e},2,1,1}$  & $3$ & $(6,2)$ &  & \\
& & & ${5,3^{\bf e},2,2}$  & $2$ & $(6,2)$ &  & \\
& & & ${5,3^{\bf e},3,1}$ & $2$ (Q)  & $(2,2)$ &  & \\
& & & ${5,3^{\bf e},4}$  & $1$ & $3$ & $E(q)$ & $(D_3,\{4^{\bf f},5,3^{\bf e})$ \\
& & & ${5,4,1^3}$ & $3$ & $(3,2)$ &  & \\
& & & ${5,4,2,1}$  & $2$ (Q)  & $(3,2)$ &  &  \\
 & & & ${5,4,3}$ & $1$ & $12$ & $E(q^4)$ & doubled dart \\
& & & ${5,5,1,1}$ & $2$ (Q)  & 2 & $E_2(q^2)$ & $(D_2, \{5,3^{\bf e}, 3^{\bf e}, 1\})$ \\
& & & ${5,5,2}$ & $1$ & $6$ & $E(q^2)$ &  $(D_2, \{5,4, 3^{\bf e}\})$ \\
\hline
$C_1$  & $1$ & $\mathbb{P}^1$ & $1^{12}$ & $10$ & $1$ & $E_{10}(q)$ & buckyballs \\
 & &  & $2,1^{10}$ & $9$ & $3$ &  & \\
 & &  & ${3,1^9}$ & $8$ & $2$ &  & \\
 & &  & ${2,2,1^8}$ & $8$ & $3$ &  & \\
 & &  & ${3,2,1^7}$ & $7$ & $6$ &  & \\
  & & & ($31$ more)  & $\cdots$ & $6$ &  &  \\
 &  &  &  $3,3,3,3$ & $2$ (Q)  & $2$ & $E_2(q^2)$ & $(D_2, \{3^{\bf e}, 3^{\bf e}, 3^{\bf e},3\})$  \\
  &  &  &  $4,3,3,2$ & $2$ (Q)  & $6$ &  & \\
    &  &  &  $4,4,2,2$ & $2$ (Q)  & $3$  &  & $(C_2,\{3^{\bf e},3^{\bf e},4,2\}) $\\
  &  &  &  $4,4,3,1$ & $2$ &$6$ &   & \\
  &  &  &  $5,3,2,2$ & $2$ & $6$ &  & \\
&  &  &  $5,3,3,1$ &  $2$ (Q)  & $2$ &  &   \\
  &  &  & $5,4,2,1$ & $2$ (Q)  &$3$ &  & \\
& & & $5,5,1,1$ & $2$ (Q)  & $1$ & $E_2(q)$ &  $(C_2, \{3^{\bf e}, 3^{\bf e},5,1\})$ \\
&  &  &  $4,4,4$ &$1$ & $3$ & $E(q)$ & $(D_3, \{4^{\bf f},5,3^{\bf e}\})$   \\
&  &  & $5,4,3$ & $1$ & $6$ & $E(q^2)$  & (doubled 30-60-90) \\
& & & $5,5,2$ & $1$ & $3$ & $E(q)$ & $(C_2, \{4,3^{\bf e},5\})$    \\
  \hline
\end{longtable}
}

\begin{centering} 

{\sc Taxonomy of triangulations} \vspace{5pt}

\end{centering}

We explain the terms used in the ``Description" column of Table \ref{bistrata}.

\begin{definition} A {\it doubled polygon} is the flat cone sphere constructed
by taking two copies of the polygon (with opposite orientations), and gluing them
along their common boundary to form a sphere, with cone points
at the vertices of the glued polygons.

Doubles appearing as $0$-dimensional bistrata are of
the (regular) {\it hexagon}, (equilateral) {\it triangle}, the {\it $\mathit{30}$-$\mathit{60}$-$\mathit{90}$ triangle},
the $60$-$120$-$60$-$120$ {\it rhombus}, the {\it pediment} which is a $30$-$30$-$120$ triangle, and
the {\it dart} which is a $60$-$90$-$120$-$90$ quadrilateral forming a fundamental domain for $\triangle/C_3$.
\end{definition}

\begin{definition}
The {\it pillowcases} are the quotients by $\{\pm 1\}$
of equilaterally triangulated elliptic curves $\C/L$, $L\subset \Z[\zeta_6]$.
\end{definition}

\begin{definition} A {\it $k$-gonal bipyramid} is the result of gluing two pyramids with $k$-gonal base
along their bases. The resulting triangulation has $2k$ triangles and $D_k$ symmetry (at least). \end{definition}
  
The hexagonal bipyramid is also a doubled hexagon, and the square bipyramid is also an octahedron.

\begin{definition} A {\it $k$-tube}, $k\in\{1,2,3,4,5,6\}$ is built from three pieces. At the left and 
right ends, we have a $k$-gonal pyramid with equilateral triangular faces. These $k$-gonal pyramids 
are glued to the two boundary components of a flat cylinder, which goes between the two ends. \end{definition}

The $k$-tubes form $1$-complex dimensional bistrata for the group $D_k$. The two real dimensions
are the length of the tube and its twisting parameter. The terminology is motivated
by {\it buckytubes}, which are the fullerenes dual to $5$-tubes.

\section{Modularity of bistratum generating functions}

In this section, we prove the modularity properties of the generating functions $f_{G,\mu}(q^{|G|})$.
We recall from \cite{es}, see also Allcock \cite{allcock},
that the lattice $(\Lambda,\star)$ is {\it $(1+\zeta_6)$-modular}:
$(1+\zeta_6)\Lambda^\vee=\Lambda$ where
$$\Lambda^\vee:=\{v\in \Lambda\otimes_{\Z[\zeta_6]}\C\,\big{|}\,v\star w\in \Z[\zeta_6]
\textrm{ for all }w\in \Lambda\}$$ is the Hermitian dual. It is the unique such Eisenstein
lattice of signature $(1,9)$. We have ${\rm Re}(v\star w) \in \tfrac{3}{2}\Z$,
and so we have an even $\Z$-lattice $(\Lambda,\cdot)$ of signature $(2,18)$ defined by
$$v\cdot w := \tfrac{2}{3}{\rm Re}(v\star w).$$ This lattice is even unimodular of signature $(2,18)$,
hence isometric to ${\rm II}_{2,18}$. Denote the $\Z$-dual of $(\Lambda,\cdot)$ by $\Lambda^*$.
Then $\Lambda^*=\Lambda$. For any even $\Z$-lattice, $\cdot$ defines a $\Q/2\Z$-valued
quadratic form on the discriminant group $\Lambda^*/\Lambda$.

\begin{definition} The {\it principal congruence subgroup
of level $N$} is the subgroup $\Gamma(N)\subset {\rm SL}_2(\Z)$ of $2\times 2$ integral,
determinant $1$ matrices which reduce to the identity mod $N$.
It is contained in the group $\Gamma_1(N)$ of matrices
which, mod $N$, are upper triangular with $1$'s on the diagonal. 
Given any subgroup $H\subset (\Z/N\Z)^*$, we define
$\Gamma_H(N)$ as the matrices whose reduction mod $N$ is upper triangular, and with diagonal
entries in $H$. \end{definition}

\begin{definition} A real-analytic function $f\colon \mathbb{H}\to \C$ is {\it modular of 
weight $(r,s)$} for a subgroup $\Gamma \subset {\rm SL}_2(\Z)$ if it satisfies 
$$f(\gamma\cdot \tau) = (c\tau +d)^r(c\overline{\tau}+d)^sf(\tau)$$ for all $\gamma = \twobytwo{a}{b}{c}{d}\in \Gamma$. \end{definition}

We will need a generalization of this to vector-valued modular forms:

\begin{definition} Let $\rho\colon {\rm SL}_2(\Z)\to {\rm GL}(E)$ be a finite-dimensional representation
on a complex vector space $E$. A real-analytic function $f\colon \mathbb{H}\to E$ is {\it modular for ${\rm SL}_2(\Z)$ of weight $(r,s)$ 
for the representation $\rho$} if $$f(\gamma\cdot \tau) = (c\tau+d)^r(c\overline{\tau}+d)^s\rho(\gamma)\cdot f(\tau)$$
 for all $\gamma\in {\rm SL}_2(\Z)$. \end{definition}

\begin{definition} A {\it modular form of weight $k$} is a function $f(\tau)$ on the upper half-plane
such that:
\begin{enumerate}
\item $f$ is holomorphic,
\item $f$ is modular of weight $(k,0)$, and
\item $f$ has bounded growth at the cusps of ${\rm SL}_2(\Z)\backslash \mathbb{H}$.
\end{enumerate}
This definition applies both to modular forms valued in a representation $\rho$ and scalar-valued
modular forms for subgroups of ${\rm SL}_2(\Z)$. \end{definition}

Thus, we distinguish between ``modular" functions and ``modular forms."

\begin{definition} Let $(\Lambda,\cdot)$ be an even $\Z$-lattice of even rank and signature $(r,s)$. The 
{\it Weil representation} (see \cite[Sec.~4]{borcherds})
$$\rho_\Lambda\colon  {\rm SL}_2(\Z)\to {\rm GL}(\C[\Lambda^*/\Lambda])$$ sends generators
 $S:= \twobytwo{0}{-1}{1}{0}$ and $T:=\twobytwo{1}{1}{0}{1}\in {\rm SL}_2(\Z)$ to 
 \begin{align*} \rho_\Lambda(S)e_\gamma &= \frac{i^{(s-r)/2}}{\sqrt{|\Lambda^*/\Lambda|}}
 \sum_{\delta\in \Lambda^*/\Lambda} e^{-2\pi i (\gamma\cdot \delta)}e_\delta \\ 
 \rho_\Lambda(T)e_\gamma &= e^{\pi i( \gamma\cdot \gamma)}e_\gamma.\end{align*} 
 Here $\gamma\in \Lambda^*/\Lambda$ and $e_\gamma$ is a basis element of the group algebra $\C[\Lambda^*/\Lambda]$. \end{definition}

\begin{definition} Let $(\Lambda,\cdot)$ be a $\Z$-lattice of signature $(2r,2s)$,
 and let $\mathbb{D}:={\rm Gr}^+(2r,V\otimes\R)$ be the Grassmannian of positive-definite $2r$-planes. 
 The {\it Siegel theta function} $\mathbb{D}\times \mathbb{H}\to \C[\Lambda^*/\Lambda]$ is
  given by $$\Theta(p,\tau):=\sum_{v\in\Lambda^*} q^{\frac{1}{2}v^+\cdot v^+}\overline{q}^{-\frac{1}{2}v^-\cdot v^-}e_{[v]}.$$
   Here $q= e^{2\pi i \tau}$, $v^+$ and $v^-$ are orthogonal projections to the positive $2r$-space $p\in \mathbb{D}$ and its negative-definite 
   perpendicular $2s$-space $p^\perp$ respectively, and $e_{[v]}$ is the basis element of $\C[\Lambda^*/\Lambda]$ 
   corresponding to the reduction of $v\in\Lambda^*$. \end{definition}

Let ${\rm O}^*(\Lambda)$ be the subgroup of the orthogonal group acting trivially
on $\Lambda^*/\Lambda$. It is clear from the definition 
that $\Theta(g\cdot p,\tau)=\Theta(p,\tau)$ for all $g\in {\rm O}^*(\Lambda)$. We
additionally have the transformation property
$$\Theta\left(p,\frac{-1}{\tau}\right)=\tau^r\overline{\tau}^s\rho_\Lambda(S)\cdot \Theta(p,\tau)$$ 
in $\tau$, which follows from the Poisson summation formula, see \cite[Thm.~4.1]{borcherds}.
Thus, in the $\tau$ variable, $\Theta(p,\tau)$ is modular of weight $(r,s)$ for the
Weil representation $\rho_\Lambda$. 

\begin{definition} Let $(\Lambda,\star)$ be an
Eisenstein lattice of signature $(1,s)$ with $\star$ valued in $(1+\zeta_6)\Z[\zeta_6]$ and let
 $v\cdot w:=\tfrac{2}{3}{\rm Re}(v\star w)$ be the corresponding even $\Z$-lattice of signature $(2,2s)$. 
 Let ${\rm U}^*(\Lambda)$ be the subgroup of the unitary group acting trivially on $\Lambda^*/\Lambda$. Define the 
 function $$g_\Lambda(\tau) := \frac{({\rm Im}\,\tau)^s}{|Z({\rm U}^*(\Lambda))|}
 \int_{{\rm U}^*(\Lambda)\backslash \mathbb{CH}^s}\Theta(p,\tau)\,dp$$ which is
a real-analytic function from $\mathbb{H}$ to $\C[\Lambda^*/\Lambda]$, modular of weight 
$(1-s,0)$ with respect to the Weil representation.\end{definition}

\begin{theorem}\label{curve-weighted-count} Let $\Gamma\subset {\rm U}^*(\Lambda)$ be a finite
index subgroup. If $s>1$, or if $s=1$ and
 $\Gamma\backslash \mathbb{CH}^s$ is compact, then
 $$f_\Lambda(\tau):=[{\rm U}^*(\Lambda):\Gamma]\left(\frac{1}{2\pi i} \partial_\tau\right)^sg_\Lambda(\tau) = 
 c_0 + \sum_{\substack{v\in \Gamma\backslash \Lambda^* \\ v\cdot v>0}}
 \frac{1}{|{\rm Stab}_\Gamma(v)|}q^{\frac{1}{2}v \cdot v}e_{[v]}$$
and furthermore, this is a modular form of weight $1+s$ for ${\rm SL}_2(\Z)$,
valued in the Weil representation $\rho_\Lambda$.

If $s=1$ and $\Gamma\backslash \mathbb{CH}^s$ is non-compact,
$f_\Lambda(\tau)$ is a {\rm quasimodular form} of weight $2$ for ${\rm SL}_2(\Z)$,
valued in $\rho_\Lambda$---the transformation law for 
$\tau\mapsto -1/\tau$ is corrected by adding a factor of the form 
 $c/({\rm Im}\,\tau)$ for some $c\in \C[\Lambda^*/\Lambda]$. \end{theorem}

\begin{proof}[Sketch.] This is essentially due to Siegel \cite{siegel}, and is proved in 
greater generality by Kudla and Millson in \cite{Kudla:1990aa}; the details are
worked out in the unimodular case for $s>1$
in \cite{es}. The vector-valued case follows the same proof. As described in 
the proof summary after Theorem \ref{weighted-count}, the integral $g_\Lambda(\tau)$
of the Siegel theta function can be shown to be a Maass form of weight $1-s$ (at least
when $s>1$) valued in the Weil representation. Applying the raising operator
$s$ times gives a modular form of weight $1+s$, and a direct computation
of the Fourier coefficients verifies the second equality in Theorem \ref{curve-weighted-count}.
The connection to \cite{Kudla:1990aa} is that Kudla-Millson's theta kernel is the result
of applying the raising operator $s$ times to the integrand $\Theta(p,\tau)$. 

The only difficulty arises in the $s=1$ case, see \cite[bottom of p.~18]{es}.
Essentially, if $\Gamma\backslash \mathbb{CH}^s$ is $1$-dimensional and has a cusp, there is 
only conditional convergence of the integral defining $g_\Lambda(\tau)$. Then,
the summation over $v\in \Lambda$ and integration over $\Gamma\backslash \mathbb{CH}^s$
fail to commute, and a more careful analysis leads to the extra factor of $c/({\rm Im}\,\tau)$.
 \end{proof}
 
 \begin{proposition}\label{modular-level} Let $f\colon \mathbb{H}\to \C[\Lambda^*/\Lambda]$ be a modular (resp. quasimodular)
 form valued in the Weil representation $\rho_\Lambda$ of an even $\Z$-lattice $\Lambda$ of even dimension.
 Let $N$ be an integer for which $\Lambda^*/\Lambda$ is $N$-torsion. Then, the $e_\gamma$-coefficient 
 of $f(\tau)$ is a  modular (resp. quasimodular) form for $\Gamma(N)$, for any $\gamma\in \Lambda^*/\Lambda$.
 \end{proposition}
%
 \begin{proof} The {\it level} of $(\Lambda,\cdot)$ is the smallest positive integer $L$
 for which $L\Lambda^*$ is an even lattice. It is well-known, see e.g.~\cite[Thm.~3.2]{ebeling}, 
 that the Weil representation $\rho_\Lambda\colon {\rm SL}_2(\Z)\to 
 {\rm GL}(\C[\Lambda^*/\Lambda])$ factors through ${\rm SL}_2(\Z/L\Z)$.
%
%
 \end{proof}

\begin{definition} Let $\kappa = \{\kappa_1,\dots,\kappa_n\}$ be a curvature profile. Define a finite group 
$$\Delta_{\kappa_i}:=\threepartdef{0}{\kappa_i=1,5,}{\Z/3\Z}{\kappa_i=2,4,}{\Z/2\Z \oplus \Z/2\Z}{\kappa_i=3}$$
and $\Delta_\kappa := \bigoplus_i \Delta_{\kappa_i}$. 
\end{definition}

An {\it Eisenstein structure} on a $\Z$-lattice is an order $3$ isometry whose only fixed point is $0$.
Such an isometry naturally endows the $\Z$-lattice with the structure of a free module over
the Eisenstein integers $\Z[\zeta_3]=\Z[\zeta_6]$, whose rank is half of the $\Z$-rank.
There are Eisenstein structures on
$\{0\}$, $A_2$, $D_4$, $E_6$, $E_8$ and we call the resulting Eisenstein lattices
$$A_0^{\Z[\zeta_6]},\,A_1^{\Z[\zeta_6]}, \,A_2^{\Z[\zeta_6]}, \,A_3^{\Z[\zeta_6]}, \,A_4^{\Z[\zeta_6]}$$ because
their Dynkin diagrams over $\Z[\zeta_6]$ have $A_n$ diagram shape:
There are $n$ generators $\{\alpha_i\}$
over $\Z[\zeta_6]$, satisfying \begin{align*} \alpha_j\star\alpha_j&=-3 \\ \alpha_j\star\alpha_{j+1} 
&= 1+\zeta_6\textrm{ for }j=1,\dots,n-1 \\  \alpha_j\star\alpha_k &= 0\textrm{ for }|j-k|\geq 2.\end{align*}
See \cite[Thm.~3]{allcock-y555} for the classification of Eisenstein (root) lattices.

We note that
the lattice $A_{10}^{\Z[\zeta_6]}$ is none other than $\Lambda$, by \cite[Sec.~5]{allcock}.
A more symmetrical diagram is a circle with $12$ nodes, see Figure \ref{dynkin},
which corresponds to presentation $\bigoplus_{i=1}^{12} \Z[\zeta_6]\alpha_i \twoheadrightarrow \Lambda$
whose kernel is a $2$-dimensional null space over $\Z[\zeta_6]$. We give a geometrical interpretation below.

\begin{figure}
\includegraphics[width=1.5in]{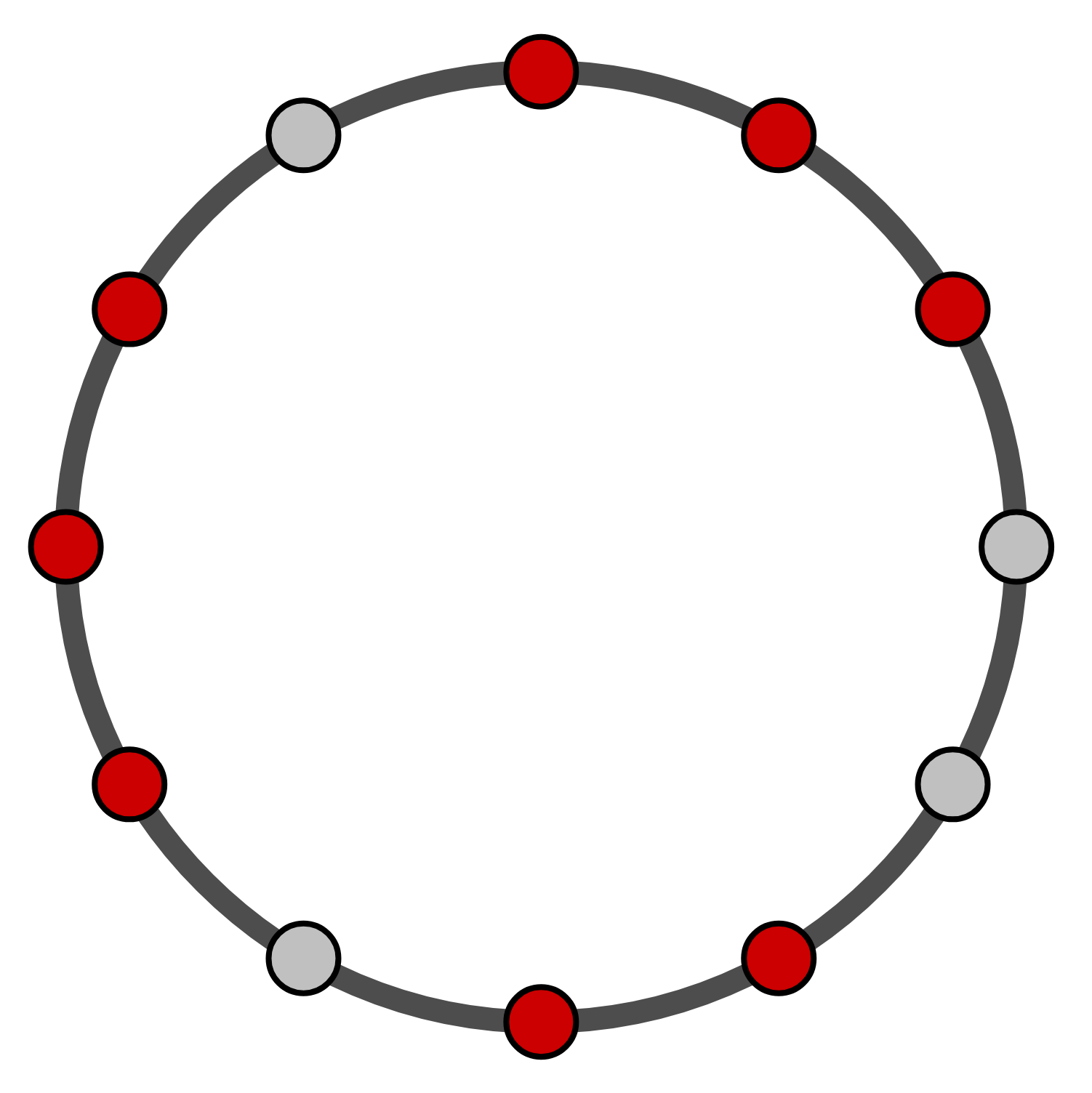}
\caption{Dynkin diagram for $\Lambda$, with 
subdiagram marked in red corresponding to curvature
profile $\kappa = \{4,4,3,1\}$.}
\label{dynkin}
\end{figure}

The discriminant groups of their underlying
 $\Z$-lattices $\{0\}$, $A_2$, $D_4$, $E_6$, $E_8$ are exactly the groups
 $\Delta_{\kappa_i}$ for $\kappa_i=1,2,3,4,5$ respectively. As the discriminant
 of an even lattice, $\Delta_{\kappa_i}$ has a $\Q/2\Z$-valued quadratic form, 
 descended from the extension of $\cdot$ to $\Lambda^*$.
 
\begin{proposition}\label{discriminant-group} Let $\Lambda_\kappa\subset \Lambda$
be the sublattice corresponding to convex
triangulations with curvature profile $\kappa$ (cf. Proposition \ref{curve-sublattice}).
The disciminant group  $\Lambda_\kappa^*/\Lambda_\kappa$ of the underlying
$\Z$-lattice $(\Lambda_\kappa,\cdot)$ is a subquotient $I^\perp/I$ associated to
an isotropic (possibly zero) subspace $I\subset \Delta_\kappa$ for the $\Q/2\Z$-valued
quadratic form on $\Delta_\kappa$.
\end{proposition}

\begin{proof} In the proof of Proposition \ref{curve-sublattice}, we construct
$\Lambda_\kappa$ as the intersection $\cap_i \Lambda_{\kappa_i}$
where $\Lambda_{\kappa_i}$ is the sublattice
for which $\kappa_i-1$ periods $\int_{p_j}^{p_{j+1}} \omega$ equal zero.

In fact, we claim $\Lambda_{\kappa_i}= (A^{\Z[\zeta_6]}_{\kappa_i-1})^\perp$ for an
embedded copy of $A^{\Z[\zeta_6]}_{\kappa_i-1}\subset \Lambda$.
The braid half-twisting $p_j, p_{j+1}$ maps to a triflection in ${\rm U}(\Lambda)$
along an Eisenstein root $\alpha_j$ satisfying $\alpha_j\star \alpha_j=-3$. Thus, the locus where
$p_j, p_{j+1}$ have collided is the perpendicular of $\alpha_j$. Since the half-twists of $p_j, p_{j+1}$
and $p_k, p_{k+1}$ commute unless $|j-k|\leq 1$, the corresponding triflections in ${\rm U}(\Lambda)$
commute, and so $\alpha_j\star\alpha_k=0$ for $|j-k|\geq 2$.
Analyzing a braid cyclically rotating $p_j,p_{j+1},p_{j+2}$ gives the formula
$\alpha_j\star \alpha_{j+1}=1+\zeta_6$. 

We conclude that $$\Lambda_\kappa = \left(\textstyle \bigoplus_i A^{\Z[\zeta_6]}_{\kappa_i-1}\right)^\perp$$
for some embedding $\phi\colon \bigoplus_i A^{\Z[\zeta_6]}_{\kappa_i-1}\hookrightarrow \Lambda$.
The Dynkin diagram in Figure \ref{dynkin}
can be understood geometrically, by placing the $12$ cone points on the equator of the sphere,
and taking $\alpha_j$ as the root whose triflection is the monodromy of the
oriented half-twist of two adjacent points on this equator.
This picture also clarifies that the embedding $\phi$
is constructed by embedding the $A_{\kappa_i-1}$ Dynkin diagrams
successively along the circle.

The Eisenstein root lattices for each coalescence are perpendicular to each other, because
they involve braiding disjoint collections of singularities. In general, $\phi$ may not
 be a primitive embedding, and $\Lambda_\kappa^\perp$ is its saturation.
 
 Finally, observe
 that $\Lambda_\kappa^\perp$ and $\Lambda_\kappa$ are mutually perpendicular, saturated
 sublattices of the unimodular $\Z$-lattice $(\Lambda,\cdot)$. So their discriminant groups
 are canonically isomorphic, via the isomorphisms $$
{\rm Disc}(\Lambda_\kappa,\cdot) \leftarrow \frac{\Lambda}{\Lambda_\kappa\oplus \Lambda_{\kappa}^\perp} \rightarrow 
{\rm Disc}(\Lambda_\kappa^\perp,\cdot).$$ Thus, the discriminant of $(\Lambda_\kappa,\cdot)$ is the subquotient
 associated to the isotropic subspace $I={\rm im}\,\Lambda_\kappa^\perp \subset \Delta_\kappa$ of the discriminant of
 $\bigoplus_i (A^{\Z[\zeta_6]}_{\kappa_i-1},\cdot)$.  \end{proof}
 
 \begin{corollary}\label{level} The discriminant group ${\rm Disc}(\Lambda_\kappa,\cdot)$ is:
 \begin{enumerate}
 \item $1$-torsion if $\kappa_i\in \{1,5\}$ for all $i$,
 \item $2$-torsion if $\kappa_i\in \{1,3,5\}$ for all $i$,
 \item $3$-torsion if $\kappa_i\in \{1,2,4,5\}$ for all $i$,
 \item $6$-torsion if $\kappa_i\in \{1,2,3,4,5\}$ for all $i$.
 \end{enumerate} \end{corollary}
 
Define $L(\kappa)\in \{1,2,3,6\}$ to be the above integer, for which $\Lambda_\kappa$
 is automatically $L(\kappa)$-torsion.

\begin{theorem}\label{weighted-kappa} The weighted generating function $f_\kappa(q)$ for triangulations
in the closure of the curvature stratum associated to $\kappa$ is a
(quasi)modular form for $\Gamma_1(L(\kappa))$
of weight ${\rm len}(\kappa)-2$. \end{theorem}

\begin{proof} This follows directly from Corollary \ref{level}, Proposition \ref{modular-level}, and
Theorem \ref{curve-weighted-count} applied to the $e_0$-coefficient, as
$$f_\kappa(q) = \frac{1}{[\Gamma_\kappa : \Gamma]}\cdot [e_0]\, f_{\Lambda_\kappa}(q).$$
Here $\Gamma= {\rm U}^*(\Lambda_\kappa)\cap \Gamma_\kappa$.  Note that we get the congruence
subgroup $\Gamma_1(L(\kappa))$ rather than $\Gamma(L(\kappa))$ because
$f_\kappa(\tau)$ is automatically a power series in $q$, rather than a fractional power of $q$.
See also Proposition \ref{substitute}.
Alternatively, $\Gamma_1(L(\kappa))$ preserves $e_0$ by \cite[Prop.~4.5]{scheithauer}.
\end{proof}

\begin{remark} This sets the constant coefficient of $f_\kappa(q)$, which was left
ambiguous in Definition \ref{completed-gen}, to be
$$c_0 = \frac{d!{\rm Vol}(\mathbb{P}\Gamma_\kappa\backslash \mathbb{CH}^d)}{6(4\pi)^d}$$
where $d={\rm len}(\kappa)-3$ is the dimension.
\end{remark}

Finally, we extend these results to the enumeration
of orbitriangulations with decorated profile $\mu$.

\begin{proposition} The Eisenstein module $\Lambda_\mu\supset \Lambda_{\mu^u}$ counting
orbitriangulations (cf. Proposition \ref{curve-sublattice}) corresponds to a partial saturation of
$\Lambda_{\mu^u}$ in $\Lambda_{\mu^u}^*$, saturating each summand
of $\Delta_\mu = \bigoplus \Delta_{\mu_i}$ with a decoration $3^{\bf e}$ or $4^{\bf f}$. \end{proposition}

\begin{proof} It was observed in the proof of Proposition \ref{curve-sublattice} that 
allowing $p_i$ to be the center of an orbi-face (decoration $4^{\bf f}$) or center
of an orbi-edge (decoration $3^{\bf e}$) corresponds to allowing the period to lie in
$\int_{p_1}^{p_i}\omega
\in \tfrac{1}{1+\zeta_6}\Z[\zeta_6]$ or $\tfrac{1}{2}\Z[\zeta_6]$, respectively.

Note that $\tfrac{1}{1+\zeta_6}\Z[\zeta_6]/\Z[\zeta_6] = \Z/3\Z$ and
$\tfrac{1}{2}\Z[\zeta_6]/\Z[\zeta_6] = \Z/2\Z\oplus \Z/2\Z$. This weakening
of period integrality produces an overlattice $\Lambda_\mu\supset\Lambda_{\mu^u}$
given by allowing the period $\int_{p_1}^{p_i}\omega$
to lie in the $\Z/3\Z$- or $(\Z/2\Z\oplus \Z/2\Z)$-enlargement.

We claim that this overlattice
corresponds to saturating $\Lambda_{\mu^u}\subset \Lambda_{\mu^u}^*$ in the
 corresponding summand of $\Delta_{\mu^u}$. To see why, first observe
 that unimodularity of $(\Lambda,\cdot)$ implies that
 $\Lambda_{\mu^u}^*$ is the image of the orthogonal projection of $\Lambda$
 into $\Lambda_{\mu^u}\otimes_\Z\Q$. In a tubular neighborhood
 of $$\Lambda_{\mu^u}\otimes_{\Z}\R = \C^{1,{\rm len}(\mu)-3}\subset \C^{1,9} =
 \Lambda\otimes_{\Z} \R,$$ the orthogonal projection has a geometric interpretation
 as the map on flat cone spheres which ``cones off'' a collection of $\mu_i$ nearby 
 singularities curvature $1$, to a single singularity of curvature $\mu_i$, see \cite[Fig.~9]{thurston}.
 
 The image of $\Lambda$ under this orthogonal projection consists exactly of the flat cone spheres
 with the weakened period integrality. For example, suppose $\gamma$ is a closed loop
 based at $p_1$ enclosing $3$ singularities of curvature $1$. Then, the period from 
 $p_1$ to the coned off point of curvature $3$ will be $\frac{1}{2}\int_{\gamma} \omega\in \frac{1}{2}\Z[\zeta_6]$.
 Analogous results hold when $1$, $2$, $4$, or $5$ singularities are coned off.
  \end{proof}
 
 \begin{theorem}\label{orbi-weighted-thm} The weighted generating function $f_{G,\mu}(q)$ for orbitriangulations
in the bistratum $(G,\mu)$ is a modular form for $\Gamma(L(\mu^u))$
of weight ${\rm len}(\mu)-2$, possibly quasimodular if ${\rm len}(\mu)=4$.   \end{theorem}

\begin{proof} The proof is the same as Theorem \ref{weighted-kappa}, except that
we sum over the coefficients $[e_\gamma] \,f_{\Lambda_{\mu}}(q)$ ranging over all
$\gamma\in \Lambda_\mu/\Lambda_{\mu^u}\subset \Lambda_{\mu^u}^*/\Lambda_{\mu^u}.$ In particular,
the resulting series may involve a fractional power of $q$, and so we only deduce
modularity for $\Gamma(L(\mu^u))$ as opposed to $\Gamma_1(L(\mu^u))$. \end{proof}

\begin{remark} The $1/|{\rm Aut}^+\mathcal{T}|$-weighted counts of generating functions for
orbitriangulations in any fixed stratum of sextic differentials (even for higher genus curves)
is a mixed weight quasimodular form for $\Gamma(6)$ by \cite{Engel:2021aa, Engel:2018ab}. While
this bounds the dimension of the space in which ${\rm BB}(q)$ lives, it does
not give sufficient control on the dimension for the purposes of this paper.
\end{remark}

\begin{proposition}\label{substitute} Let $h(q)$ be a quasimodular form for $\Gamma(N)$ which is a power
series in $q^{1/d}$ for some divisor $d\mid N$. Then $h(q^d)\in \C[[q]]$ is a quasimodular
form for $\Gamma_H(Nd)$ where $H\subset (\Z/Nd\Z)^*$ is the subgroup of units which are $1$
mod $N$.   \end{proposition}

\begin{proof} Because $h(q)$ is a series in $q^{1/d}$, it is modular for
 any matrix $$M=\twobytwo{1+a_1N}{a_2d}{a_3N}{1+a_4N}\in \Gamma_1(N).$$ 
 Substituting $q\to q^d$ corresponds to the substitution $\tau\mapsto d\tau$,
which is induced by the element $$D = \twobytwo{d}{0}{0}{1}\in {\rm GL}_2^+(\Q).$$ 
Then $h(q^d)$ is modular for any matrix $$D^{-1}MD= \twobytwo{1+a_1N}{a_2}{a_3dN}{1+a_4N}$$
which exactly defines the group $\Gamma_H(Nd)$. \end{proof}

\begin{definition} Define $\Gamma(N,d)$ to be $\Gamma_H(Nd)$ where $H$ is the subgroup
of $(\Z/Nd\Z)^*$ of units congruent to $1$ mod $N$. We say that a (quasi)modular form
for $\Gamma(N,d)$ is of {\it level} $(N,d)$. \end{definition}

\begin{theorem}\label{useful} For any bistratum $(G,\mu)$,
its generating function for triangulations
$f_{G,\mu}(q^{|G|})$ is a modular form (quasimodular when
${\rm len}(\mu)=4$) for $\Gamma(L(\mu), |G|)$ of weight ${\rm len}(\mu)-2$.
\end{theorem}

This follows directly from Theorem \ref{orbi-weighted-thm} and Proposition \ref{substitute}.

\begin{definition} Let $M_k^{N,d}=M_k(\Gamma(N,d))$ denote the space of modular forms 
of level $(N,d)$ and weight $k$ and let $M_k^N$ denote
the space of modular forms for $\Gamma_1(N)=\Gamma(N,1)$. \end{definition}

We have $M_k^{N,d}\supset M_k^{N',d'}$ when $N'\mid N$ and $d'\mid d$. 

\begin{proposition} The generating function ${\rm BB}(q)$ for oriented fullerenes lies in the
$152$-dimensional vector space
\begin{align*} 
\mathbb{O}:=M_{10}^1\,\oplus & \,M_9^3\oplus M_8^6\oplus M_7^6\oplus M_6^{6,2}\oplus M_5^{6,2} \oplus M_4^{6,12} \oplus
M_3^{6,12} \oplus M_2^{6,12}\, \oplus   \\
& \C E_2(q) \oplus \C E_2(q^5) \oplus_{n\in \{1,2,3,4,5,6,10\}} \C E(q^n) \oplus \C.
\end{align*} \end{proposition}

\begin{proof} For all bistrata of dimension $2$ (weight $3$) or greater,
we simply go through the list in Table \ref{bistrata}, applying Theorem \ref{useful}
in each case, to determine a space of modular forms containing $f_{G,\mu}(q^{|G|})$. The results are listed
in the level and weight columns of Table \ref{bistrata}. Using the containments
between $M_k^{N,d}$ spaces, we take the least common multiples of the $N$'s and the $d$'s
to produce a level $(N,d)$ in that weight containing all $f_{G,\mu}(q^{|G|})$.

In weight $2$, the series $f_{G,\mu}(q^{|G|})$ may be only a quasimodular form, indicated in
the weight column of Table \ref{bistrata} by (Q). Any quasimodular form
of weight $2$ is modular after subtracting a constant multiple of
 $E_2(q)=-\tfrac{1}{24}+\sum_{n\geq 1} \sigma(n)q^n$. So $M_2^{6,12}+\C E_2(q)$ 
covers all the $1$-dimensional bistrata for which $f_{G,\mu}(q^{|G|})$ is not explicitly listed.

The bistrata of dimension $1$ where $f_{G,\mu}(q^{|G|})$ is given explicitly are the
$k$-tubes. Any $k$-tube is a $C_k$-cover of a triangulation in the curvature stratum
$\kappa =\{5,5,1,1\}$, branched over the two points of curvature $5$. We necessarily
have $f_{G,\{5,5,1,1\}}(q)\in \C E_2(q)$ for any $G$ because it is quasimodular of level $1$ and weight $2$.
Thus, the generating function for $k$-tubes is $E_2(q^k)$, up to a constant. Note $E_2(q) - kE_2(q^k)$
is modular of level $k$ and so lies in $M_2^{6,12}$ for $k=1,2,3,4,6$. So
the only extra series we need beyond $E_2(q)$ is $E_2(q^5)$.

The bistrata of dimension $0$ correspond to flat cone spheres unique up to scaling.
If the smallest triangulation in this bistratum has $2n$ triangles, then the generating
function for the stratum is $E(q^n)$ because the corresponding Eisenstein lattice has rank $1$
(Ex.~\ref{gen-ex}). Finally, we add a constant in $\C$. Then, by Theorem \ref{include-exclude},
the resulting space of forms contains ${\rm BB}(q)$.

The spaces in the decomposition are mutually linearly independent, since for weight $\geq 3$,
there is only one space of that weight, and for weight $2$, the only generator beyond 
 $M^{6,12}_2\oplus \C E_2(q)$ is $E_2(q^5)$, which does not lie in the former space. The weight $1$
 terms $E(q^n)$ are also linearly independent.
\end{proof}

\begin{proposition} \label{full-rank}
The map $\C[[q]]\to \C[[q]]/q^{200}\simeq \C^{200}$ which truncates at $q^{200}$
is injective upon restriction to $\mathbb{O}$. \end{proposition}

\begin{proof} It is an easy check in SAGE that the $200\times 152$ coefficient matrix
of a basis of the above space has maximal rank. \end{proof}

Using the first $200$ coefficients of ${\rm BB}(q)$ from Appendix \ref{appendix}, computed
using {\it buckygen}, we can solve a linear system to determine ${\rm BB}(q)\in \mathbb{O}$ explicitly. 
Proposition \ref{full-rank} proves that the computed solution is unique.
We find that ${\rm BB}(q)$ lies in a much smaller space:

\begin{proposition} ${\rm BB}(q)$ lies in the $47$-dimensional space \begin{align*} &\mathbb{M} =
\mathbb{E}_{10}^1\oplus \mathbb{E}_9^3\oplus (\mathbb{E}_8^2+_{\mathbb{E}_8^1}\mathbb{E}_8^3) \oplus
\mathbb{E}_7^6\oplus \mathbb{E}_6^6\oplus \mathbb{E}_5^6 \oplus (\mathbb{E}_4^6 +_{\mathbb{E}_4^2} \mathbb{E}_4^{2,2})
\,\oplus \\
(\mathbb{E}_3^{6,2}&+_{\mathbb{E}_3^3}\mathbb{E}_3^{3,3})\oplus (\mathbb{E}_2^{6,2}+_{\mathbb{E}_2^3}\mathbb{E}_2^{3,3}) \oplus  \C E_2(q) \oplus
\C E_2(q^5) \oplus_{n
\leq 6} \C E(q^n) \oplus \C\end{align*}
where $\mathbb{E}_k^{N,d}\subset M_k^{N,d}$ is the Eisenstein submodule.
\end{proposition}

Furthermore, all but a small number of coefficients in the expression of ${\rm BB}(q)$
in the SAGE basis of $\mathbb{M}$ are nonzero,
suggesting that this result is nearly optimal. We now give an explicit formula.

If $\chi$ and $\psi$ are two Dirichlet characters of modulus $L$ and $R$ respectively, with $\psi$ primitive, and $k$ is a positive integer such that $\chi(-1)\psi(-1) = (-1)^k$, the Eisenstein series $E_{k, \chi, \psi}$ is defined by its Fourier expansion
\begin{equation}\label{fourier}
E_{k, \chi, \psi}(q) = c_0 + \sum_{n \geq 1} \left(\sum_{d|n} \chi(n/d) \psi(d) d^{k-1} \right) q^n.
\end{equation}
where $c_0$ is 0 if $L > 1$ and a certain Bernoulli number if $L=1$. If both characters are the trivial character $\triv$, we abbreviate $E_k = E_{k, \triv, \triv}$. If $k\neq 2$, then for any positive integer $t$, and any $N$ such that $LRt$ divides $N$, the series $E_{k,\chi,\psi}(q^t)$ is a modular form of weight $k$ for $\Gamma_1(N)$ \cite[Thm. 5.8]{Stein:2007aa}. If $k = 2$, then $E_2(q) - t E_2(q^t)$ is a modular form of weight 2 for $\Gamma_0(t)$.

Let  $\chi_3$ denote the non-trivial Dirichlet character with modulus 3, and $\triv$ the trivial character with modulus 1.

\begin{example}\label{eis-ex} We have $E(q)=\tfrac{1}{6}\sum_{v\in \Z[\zeta_6]} q^{v\overline{v}}= E_{1,\triv,\chi_3}(q)$. \end{example}

We may now state our main theorem:

\begin{theorem}\label{main}
${\rm BB}(q)=$
\begin{align*}  &\,\,\,\, \,\,\,\tfrac{809}{2^{15} 3^{13} 5^2}E_{10}(q)
-\tfrac{1}{2^{16} 3^{11} 5^2}E_{9,\triv, \chi_3}(q)
-\tfrac{1}{2^{16} 3^3 5^2}E_{9,\chi_3,\triv}(q)
+ \tfrac{515}{2^{13} 3^{11}}E_8(q)  \\ &
- \tfrac{1}{2^6 3^9 5}E_8(q^2)
-\tfrac{1}{2^{11} 3^4 5}E_8(q^3)
+\tfrac{47}{2^{14} 3^{12}}E_{7,\triv,\chi_3}(q)
+ \tfrac{1}{2^8 3^9 5\cdot 7}E_{7,\triv,\chi_3}(q^2) \\ &
-\tfrac{7}{2^{14} 3^3 5}E_{7,\chi_3,\triv}(q) 
-\tfrac{1}{2^8 3^3 5\cdot 7}E_{7,\chi_3,\triv}(q^2) 
+\tfrac{2138939}{2^{14} 3^{11} 5^2}E_6(q)
+\tfrac{3571}{2^9 3^8 5}E_6(q^2)  \\ &
-\tfrac{29}{2^{11} 3^6 5} E_6(q^3)
+\tfrac{1}{2^6 3^3 5}E_6(q^6) 
+\tfrac{274421}{2^{15} 3^{11} 5^2}E_{5,\triv,\chi_3}(q)
 -\tfrac{8231}{2^{11} 3^9 5}E_{5,\triv,\chi_3}(q^2) \\ &
 -\tfrac{2393}{2^{15} 3^2 5^2}E_{5,\chi_3,\triv}(q)
  -\tfrac{203}{2^{11} 3^2 5}E_{5,\chi_3,\triv}(q^2)
  +\tfrac{12074417}{2^{13} 3^{12}}E_4(q) 
+\tfrac{287281}{2^{10} 3^9 5}E_4(q^2) \\ &
-\tfrac{2672137}{2^11 3^9 5}E_4(q^3)
-\tfrac{1}{2^2 3^3}E_4(q^4)
-\tfrac{4289}{2^8 3^6}E_4(q^6)
+ \tfrac{3343037}{2^{14} 3^{11}} E_{3,\triv,\chi_3}(q) \\ &
+\tfrac{227779}{2^{12} 3^8 5}E_{3,\triv,\chi_3}(q^2)  
 +\tfrac{1}{2^3 3^3}E_{3,\triv,\chi_3}(q^3)
+\tfrac{1}{2^5 3^2}E_{3,\triv,\chi_3}(q^4) 
-\tfrac{33103}{2^{14} 3^3 5}E_{3,\chi_3,\triv}(q)  \\ &
-\tfrac{713}{2^{12} 3\cdot 5}E_{3,\chi_3,\triv}(q^2)
 +\tfrac{1}{2^3 3}E_{3,\chi_3,\triv}(q^3) 
-\tfrac{1}{2^5 3}E_{3,\chi_3,\triv}(q^4)
+\tfrac{4314057521}{2^{15}3^{11}5^2} E_2(q)  \\ &
- \tfrac{59136661}{2^{14} 3^9 5} E_2(q^2)
- \tfrac{32541151}{2^{11} 3^{10} 5}E_2(q^3)
+ \tfrac{1}{2^5 3}E_2(q^4)
+ \tfrac{2}{3\cdot 5} E_2(q^5)
+\tfrac{5023687}{2^{10} 3^8 5}E_2(q^6)   \\ &
+\tfrac{1}{2^6}E_2(q^9)
+\tfrac{2}{3^3}E_2(q^{12})
-\tfrac{24019585289}{2^{15} 3^{13} 5^2} E_{1,\triv,\chi_3}(q)
-\tfrac{3212143}{2^{11} 3^{9} 5} E_{1,\triv,\chi_3}(q^2)  \\ &
-\tfrac{1}{2^2 3^4} E_{1,\triv,\chi_3}(q^3)
-\tfrac{7}{2^7 3^3} E_{1,\triv,\chi_3}(q^4)
-\tfrac{2}{3 \cdot 5} E_{1,\triv,\chi_3}(q^5) 
 -\tfrac{1}{2\cdot 3^2} E_{1,\triv,\chi_3}(q^6)   \\ &
 + \tfrac{3608212332449}{2^{18} 3^{14} 5\cdot 11}.
 \end{align*} 
\end{theorem}

\begin{example} If $n$ is congruent to $1$ mod $3$, and not divisible by $2$ or $5$,
the number $\mathrm{BB}_n$ of fullerenes with $2n$ carbon atoms is
$\sum_{d|n} \chi_3(d) p(\chi_3(d) d)$, where $p(d)$ is the polynomial
T\begin{align*}
p(d) =\frac{1}{2^{15}3^{13}5^2} \big(809d^9& - 29529 d^8 + 463500 d^7 - 4126380 d^6  \\
+3&8500902 d^5  - 421442982 d^4+ 3622325100 d^3  \\
& - 18042623820 d^2 + 38826517689 d - 24019585289\big).
\end{align*}
In particular, if $n$ is prime and congruent to $1$ mod $3$, then $\mathrm{BB}_n=p(n)$.
\end{example}

 \begin{example} There are $1203397779055806181762759
 $ fullerenes with $10000$ carbon atoms. 
 This trails the leading term $\frac{809}{2^{15}3^{13}5^2}\sigma_9(5000)$ by about $0.73\%$.
 \end{example}
 
 \begin{remark} \label{siegel-weil}
 The Siegel-Weil formula (\cite{siegel}, \cite{Weil:1965aa}) can be used to express each 
 completed generating function $f_{G,\mu}(q^{|G|})$ as a certain Eisenstein series. In more 
 detail, if $\Lambda$ is any Hermitian Eisenstein lattice of signature $(1,n-1)$, $n \geq 2$, and 
 $f_\Lambda(\tau)$ is the generating function as defined in Theorem \ref{curve-weighted-count}, 
 Siegel-Weil can be used to derive an explicit expression for $f_\Lambda$ as a certain 
 holomorphic Eisenstein series.  The precise answer depends at least on
 the discriminant 
 group of the lattice $\Lambda$ and the index of $\Gamma$ in $U^*(\Lambda)$. 
 In principle, these could be computed for each bistratum and used 
 to derive the formula for ${\rm BB}(q)$ without 
 counting any fullerenes at all, though the computation
appears long and difficult. We note that the reduction 
 from the adelic integral in \cite[Theorem 5]{Weil:1965aa} to our expression
  for $f_\Lambda$ uses that every indefinite Hermitian Eisenstein lattice 
  has unitary class number 1 \cite[Theorem 5.24]{Shimura:1964aa}.
 
The difficulty of the computation notwithstanding, the Siegel-Weil formula 
explains some apparent patterns in the formula for ${\rm BB}(q)$ in Theorem 
\ref{main}. For one, it is the reason that every term is an Eisenstein series.
But it also tells us, using almost nothing about the lattice $\Lambda$, 
which automorphic representation the Eisenstein series must lie in. Namely, 
if the dimension of $\Lambda$ is even, then $f_\Lambda$ lies in the 
adelic principle series representation induced from the trivial Hecke 
character, and if the dimension of $\Lambda$ is odd, then it lies in the
 principle series representation induced from the non-trivial 
 Hecke character $\chi_3$ mod 3 (\cite{Kudla:1994aa}, see also \cite{Yamana:2011aa}). 
 In terms of the Eisenstein series $E_{k,\chi,\psi}(q^t)$, this explains why 
 only $E_{k,1,1}$ appears in the formula for ${\rm BB}(q)$ when $k$ is even 
 and only $E_{k,1, \chi_3}$ and $E_{k, \chi_3,1}$ appear when $k$ is odd. 

Using these restrictions, together with the bounds on $t$ that can 
appear in $E_{k,\chi,\psi}(q^t)$ from the level of the lattice, one could 
compute each completed generating function by finding only a
few Fourier coefficients by hand, and then arrive at the final answer 
by inclusion-exclusion. We give an example of such a computation below, 
which gives some indication as to why this was not the method chosen.
 \end{remark}
 
  \begin{example} By Remark \ref{siegel-weil}, each weighted generating
  function $f_{G,\mu}(q)$ lies in a space of Eisenstein series.
  The primary difficulty to explicitly implementing the inclusion-exclusion to compute ${\rm BB}(q)$
  is to count triangulations with correct weights. For each of the 96 triangulable bistrata, the weight
  of a triangulation depends on the ambient bistratum, and the
  mass given in \ref{mass} is no longer valid.
  We illustrate this with the computation of $f_{C_1, \{2,1^{10}\}}(q)$.
  
  There are two triangulations
  with two triangles, whose curvature profiles are $(4,4,4)$ and $(5,5,2)$---the doubled triangle
  and doubled pediment. Both triangulations
  lie in the closure of the bistratum $(C_1, \{2,1^{10}\})$ but the triangulation $\mathcal{T}$ with profile $(5,5,2)$ 
  in $M$ breaks into two orbits $\mathcal{T}_1$, $\mathcal{T}_2$ under $\Gamma_{C_1, \{2,1^{10}\}}$ since the profile $(5,5,2)$
  can arise as two different coalescences: \begin{align*} 
  \mathcal{T}_1\colon (5,5,2)&= (1+1+1+1+1, &&1+1+1+1+1, &&& 2), \\
    \mathcal{T}_2\colon(5,5,2)&= (2+1+1+1, && 1+1+1+1+1, &&& 1+1). 
  \end{align*}
  The order of the stabilizer in $\Gamma_{C_1, \{2,1^{10}\}}$ of the corresponding triangulation
  can be computed via the formula  \cite[Prop.~3.6]{thurston} for the complex link fraction of a coalescence of cone points,
  divided by the subgroup of the automorphism group of the triangulation preserving the structure of the coalescence.
  So, \begin{align*} 
  {\rm Stab}_{\Gamma_{C_1, \{2,1^{10}\}}}(\mathcal{T}_1) &= 
  \frac{1}{2}\cdot \frac{(1-\tfrac{5}{6})^4}{5!} 
  \cdot \frac{(1-\tfrac{5}{6})^4}{5!} \cdot \frac{(1-\tfrac{2}{6})^0}{1!}=\frac{1}{2^{15}3^{10}5^2}, \\
  {\rm Stab}_{\Gamma_{C_1, \{2,1^{10}\}}}(\mathcal{T}_2) &=
  \frac{1}{1}\cdot \frac{(1-\tfrac{5}{6})^3}{1!\cdot 3!} 
  \cdot \frac{(1-\tfrac{5}{6})^4}{5!} \cdot \frac{(1-\tfrac{2}{6})^1}{2!}=\frac{1}{2^{11}3^{10}5}. \end{align*}
 On the other hand, the triangulation $\mathcal{T}'$ with profile $(4,4,4)$ corresponds to a single
 $\Gamma_{C_1, \{2,1^{10}\}}$-orbit with coalescence
 \begin{align*} 
   \mathcal{T}'\colon (4,4,4)&= (2+1+1, &&1+1+1+1, &&& 1+1+1+1), \end{align*}
  and the appropriate weight is 
 \begin{align*} 
  {\rm Stab}_{\Gamma_{C_1, \{2,1^{10}\}}}(\mathcal{T}') &= 
  \frac{1}{2}\cdot \frac{(1-\tfrac{4}{6})^2}{1!\cdot 2!} 
  \cdot \frac{(1-\tfrac{4}{6})^3}{4!} \cdot \frac{(1-\tfrac{4}{6})^3}{4!}=\frac{1}{2^83^{10}}. \end{align*}
  Combining these weights, we get the $q^1$-coefficient
  $$[q^1] \,f_{C_1, \{2,1^{10}\}}(q) = \frac{1}{2^{15}3^{10}5^2} +
\frac{1}{2^{12}3^{10}5}+  \frac{1}{2^83^{10}} = \frac{17\cdot 193}{2^{15}3^{10}5^2}.$$

Performing an analogous calculation for the $\Gamma_{C_1, \{2,1^{10}\}}$-orbits of triangulations 
with $4$ triangles, given by curvature coalescences \vspace{5pt}

\begin{centering}

 \begin{tabular}{lllr}
$(3,3,3,3)= (2+1$, & $1+1+1$, & $1+1+1$, & $1+1+1$), \\
$(5,4,3,0)= (2+1+1+1$, & $1+1+1+1$, & $1+1+1$, & $\emptyset$), \\
$(5,4,3,0)= (1+1+1+1+1$, & $2+1+1$, & $1+1+1$, & $\emptyset$), \\
$(5,4,3,0)= (1+1+1+1+1$, & $1+1+1+1$, & $2+1$, & $\emptyset$), \\
$(5,5,1,1)= (2+1+1+1$, & $1+1+1+1+1$, & $1$, & $1$), \\
$(4,4,2,2)= (2+1+1$, & $1+1+1+1$, & $1+1$, & $1+1$), \\
$(4,4,2,2)= (1+1+1+1$, & $1+1+1+1$, & $1+1$, & $2$), \\
  \end{tabular}
  
  \end{centering}
  \vspace{5pt}
  
\noindent we find that the combined mass $[q^2]\, f_{C_1, \{2,1^{10}\}}(q)$  of these $7$ orbits is
$$\frac{1}{2^{10}3^4}+
 \frac{1}{2^{10}3^9}+ 
  \frac{1}{2^{11}3^85}+ 
    \frac{1}{2^{11}3^95}+ 
     \frac{1}{2^{11}3^95}+ 
      \frac{1}{2^{5}3^8}+ 
            \frac{1}{2^{7}3^9} =  \frac{17\cdot 41}{2^{11}3^9}.$$

Note that $f_{C_1, \{2,1^{10}\}}(q)$ has level $3$ and weight $9$
by Theorem \ref{orbi-weighted-thm}. Then
Remark \ref{siegel-weil} implies that $$f_{C_1, \{2,1^{10}\}}(q)= 
AE_{9,{\bf 1},\chi_3}(q) + BE_{9,\chi_3, {\bf 1}}(q)$$
for some constants $A,B\in \C$. Using Equation \ref{fourier} for
the Fourier coefficients, we find that the unique solution
giving $[q^1]$, $[q^2]$ correctly is $$A= \frac{1}{2^{16} 3^{10} 5^2}, \quad
B = \frac{1}{2^{16} 3^2 5^2}.$$

This answer agrees with the linear solve from Theorem \ref{main}, since
the generic orbifold order along the completed bistratum $\overline{M}_{(C_1, \{2,1^{10}\})}\subset M$
is $3$. Hence the first correction in the inclusion-exclusion is
$${\rm BB}(q) = f_{C_1, \{1^{12}\}}(q) - \tfrac{1}{3}  f_{C_1, \{2,1^{10}\}}(q) +
(\textrm{terms of weight}\leq 8),$$ confirming that the weighted count was performed correctly.
\end{example}
 
 \begin{remark} One strange occurrence is the non-appearance of the generating
 function $ E_{1,\triv,\chi_3}(q^{10})$ for the icosahedral bistratum. Essentially, it means that the
 various inclusions and exclusions of this bistratum end up cancelling perfectly.
 We do not know why this is the case.
 \end{remark}

\appendix
 \section{{\it buckygen} counts (by Jan Goedgebeur)}
 \label{appendix}

In~\cite{BGM12} an algorithm is described to generate all pairwise non-isomorphic fullerenes 
up to a given order $n$. An efficient implementation of this algorithm was made in the program 
called \textit{buckygen} (of which the source code is available at~\cite{{buckygen-site}}). This 
program was already used in~\cite{BGM12} to generate all 
pairwise non-isomorphic (unoriented) 
fullerenes up to 400 vertices. These counts were independently confirmed by the generator 
\textit{fullgen} of Brinkmann and Dress~\cite{BD97} up to 380 vertices.

By default, 
\textit{buckygen} outputs or counts one graph from each isomorphism class. However, there 
is also a built-in option to write one member of each orientation-preserving isomorphism class
 (by using the flag \verb|-o|). In this case, the output graphs are tested for the presence of an 
 orientation-reversing automorphism. If there is none, the mirror image of the graph is output 
 as well. We used this option to generate all oriented fullerenes (i.e.\ considering enantiomorphic 
 fullerenes as distinct) up to 400 vertices. The computation took approximately 3 CPU years 
 and was performed on the supercomputer of the VSC (Flemish Supercomputer Center). 
 The counts are found in Table~\ref{table:counts_fullerenes}. Previously these counts were
  determined up to 98 vertices (cf.\ sequence \href{https://oeis.org/A057210}{A057210}
   in the \textit{On-Line Encyclopedia of Integer Sequences}~\cite{OEIS}).
   
   \newpage

\begin{centering}

\begin{table}[H]
\caption{Counts $\mathrm{BB}_n$ of oriented fullerenes with $2n$ vertices} \label{table:counts_fullerenes} 
\begin{small}
\begin{tabular}{|cc||cc||cc||cc|}
\hline
$2n$ & $\mathrm{BB}_n$ & $2n$ & $\mathrm{BB}_n$ & $2n$ & $\mathrm{BB}_n$ & $2n$ & $\mathrm{BB}_n$ \\
\hline
0-18 &	0 &	114 &	2 014 713 &	210 &	672 960 919 &	306 &	22 328 857 779\\
20 &	1 &	116 &	2 411 814 &	212 &	739 119 987 &	308 &	23 803 804 599\\
22 &	0 &	118 &	2 814 401 &	214 &	803 025 838 &	310 &	25 177 747 311\\
24 &	1 &	120 &	3 345 147 &	216 &	880 381 442 &	312 &	26 820 388 881\\
26 &	1 &	122 &	3 882 755 &	218 &	954 791 893 &	314 &	28 342 433 015\\
28 &	3 &	124 &	4 588 131 &	220 &	1 045 149 256 &	316 &	30 170 072 140\\
30 &	3 &	126 &	5 297 876 &	222 &	1 131 745 304 &	318 &	31 860 956 489\\
32 &	10 &	128 &	6 224 194 &	224 &	1 236 558 914 &	320 &	33 883 714 816\\
34 &	9 &	130 &	7 157 006 &	226 &	1 337 268 127 &	322 &	35 760 178 173\\
36 &	23 &	132 &	8 359 652 &	228 &	1 458 771 168 &	324 &	38 003 821 487\\
38 &	30 &	134 &	9 570 462 &	230 &	1 575 048 577 &	326 &	40 074 384 246\\
40 &	66 &	136 &	11 128 035 &	232 &	1 715 797 098 &	328 &	42 560 801 885\\
42 &	80 &	138 &	12 683 755 &	234 &	1 850 016 768 &	330 &	44 852 179 762\\
44 &	162 &	140 &	14 676 481 &	236 &	2 011 965 672 &	332 &	47 592 925 209\\
46 &	209 &	142 &	16 671 248 &	238 &	2 166 827 672 &	334 &	50 126 102 754\\
48 &	374 &	144 &	19 201 153 &	240 &	2 353 180 355 &	336 &	53 155 439 383\\
50 &	507 &	146 &	21 728 036 &	242 &	2 530 571 274 &	338 &	55 939 718 941\\
52 &	835 &	148 &	24 930 330 &	244 &	2 744 801 058 &	340 &	59 284 163 396\\
54 &	1 113 &	150 &	28 109 625 &	246 &	2 948 134 826 &	342 &	62 354 562 777\\
56 &	1 778 &	152 &	32 122 355 &	248 &	3 192 869 772 &	344 &	66 028 033 946\\
58 &	2 344 &	154 &	36 112 223 &	250 &	3 425 774 760 &	346 &	69 410 105 709\\
60 &	3 532 &	156 &	41 107 620 &	252 &	3 705 433 966 &	348 &	73 456 137 608\\
62 &	4 670 &	158 &	46 064 096 &	254 &	3 970 400 266 &	350 &	77 160 823 316\\
64 &	6 796 &	160 &	52 272 782 &	256 &	4 289 774 460 &	352 &	81 612 318 561\\
66 &	8 825 &	162 &	58 393 248 &	258 &	4 591 482 273 &	354 &	85 683 960 523\\
68 &	12 501 &	164 &	66 032 535 &	260 &	4 953 928 565 &	356 &	90 556 797 909\\
70 &	16 091 &	166 &	73 582 782 &	262 &	5 297 277 909 &	358 &	95 026 883 784\\
72 &	22 142 &	168 &	82 940 953 &	264 &	5 708 945 811 &	360 &	100 377 559 914\\
74 &	28 232 &	170 &	92 160 881 &	266 &	6 097 098 730 &	362 &	105 257 194 470\\
76 &	38 016 &	172 &	103 602 394 &	268 &	6 564 284 155 &	364 &	111 124 529 629\\
78 &	47 868 &	174 &	114 816 693 &	270 &	7 003 726 743 &	366 &	116 472 010 663\\
80 &	63 416 &	176 &	128 686 912 &	272 &	7 530 784 730 &	368 &	122 874 832 878\\
82 &	79 023 &	178 &	142 308 166 &	274 &	8 027 877 473 &	370 &	128 726 803 405\\
84 &	102 684 &	180 &	159 057 604 &	276 &	8 623 164 190 &	372 &	135 735 763 019\\
86 &	126 973 &	182 &	175 455 386 &	278 &	9 181 941 021 &	374 &	142 104 866 311\\
88 &	162 793 &	184 &	195 657 995 &	280 &	9 853 806 446 &	376 &	149 768 451 649\\
90 &	199 128 &	186 &	215 335 951 &	282 &	10 482 937 182 &	378 &	156 727 483 454\\
92 &	252 082 &	188 &	239 498 752 &	284 &	11 236 734 954 &	380 &	165 065 400 161\\
94 &	306 061 &	190 &	263 096 297 &	286 &	11 944 678 343 &	382 &	172 658 723 068\\
96 &	382 627 &	192 &	291 923 618 &	288 &	12 791 770 724 &	384 &	181 761 605 787\\
98 &	461 020 &	194 &	319 971 240 &	290 &	13 583 361 145 &	386 &	190 001 957 648\\
100 &	570 603 &	196 &	354 321 904 &	292 &	14 534 386 155 &	388 &	199 925 650 318\\
102 &	682 017 &	198 &	387 597 455 &	294 &	15 421 369 613 &	390 &	208 906 495 800\\
104 &	836 457 &	200 &	428 220 000 &	296 &	16 483 227 350 &	392 &	219 673 875 162\\
106 &	993 461 &	202 &	467 658 270 &	298 &	17 476 267 675 &	394 &	229 445 258 636\\
108 &	1 206 782 &	204 &	515 596 902 &	300 &	18 663 927 061 &	396 &	241 169 817 560\\
110 &	1 424 663 &	206 &	561 974 100 &	302 &	19 768 989 050 &	398 &	251 745 893 960\\
112 &	1 718 034 &	208 &	618 503 629 &	304 &	21 096 191 945 &	400 &	264 495 159 034\\
\hline
\end{tabular}
\end{small}
\end{table}

\end{centering}

\end{document}